\documentclass[11pt]{article}

\usepackage[margin = 1in]{geometry}

\usepackage[shortlabels]{enumitem}
\setlist{nolistsep, noitemsep}

\usepackage{amssymb}
\usepackage{amsthm}
\usepackage{amsmath}
\usepackage{hyperref}

\usepackage{titlesec}
\titleformat{\subsubsection}[runin]{\normalfont\normalsize\bfseries}{\thesubsection.}{1pt}{}

\newtheorem{thm}{Theorem}[section]
\newtheorem{lemma}[thm]{Lemma}
\newtheorem{prop}[thm]{Proposition}
\newtheorem{conj}[thm]{Conjecture}

\theoremstyle{definition}
\newtheorem{define}[thm]{Definition}

\usepackage{thmtools}
\usepackage{thm-restate}

\DeclareMathOperator{\Save}{Save}
\DeclareMathOperator{\Gap}{Gap}
\DeclareMathOperator{\ch}{ch}

\DeclareMathOperator{\ad}{ad}
\DeclareMathOperator{\mad}{mad}

\newcommand{\eps}{\ensuremath{\varepsilon}}

\newcommand{\Prob}[1]{\ensuremath{%
    \mathbb P\left[#1\right]
}}
\newcommand{\ProbCond}[2]{\ensuremath{%
    \mathbb P\left[#1\:\middle|\:#2\right]
  }}
\newcommand{\Expect}[1]{\ensuremath{%
    \mathbb E\left[#1\right]
  }}

\newcommand{\savings}{\ensuremath{\ensuremath{\mathbf{savings}}}}
\newcommand{\aberrance}{\ensuremath{\mathbf{unmatched}}}
\newcommand{\totesabs}{\ensuremath{%
    \mathbf{unmatched}^{\mathrm{tot}}
  }}
\newcommand{\pairs}{\ensuremath{\mathbf{pairs}}}
\newcommand{\totespairs}{\ensuremath{%
    \mathbf{pairs}^{\mathrm{tot}}
  }}

\newcommand{\trips}{\ensuremath{\mathbf{trips}}}
\newcommand{\totestrips}{\ensuremath{%
    \mathbf{trips}^{\mathrm{tot}}
  }}

\newcommand{\logexp}{\ensuremath{10}}

\newcommand{\aberrant}{\ensuremath{\mathrm{Lord}}}
\newcommand{\slaberrant}{\ensuremath{\mathrm{Slightly\mbox-Lord}}}
\newcommand{\egal}{\ensuremath{\mathrm{Egal}}}

\newcommand{\subserv}{\ensuremath{\mathrm{Subserv}}}

\newcommand{\unact}{\ensuremath{\mathbf{inactive}}}

\newcommand{\uncolvtcs}{\ensuremath{U}}

\newcommand{\ordering}{\ensuremath{\prec}}

\title{On the density of critical graphs with no large cliques}

\author{Tom Kelly\thanks{
    School of Mathematics, Georgia Institute of Technology, Atlanta 30332, USA.
    Email: \texttt{tom.kelly@gatech.edu}.  Partially supported by the EPSRC, grant no. EP/N019504/1.}
  \and
  Luke Postle\thanks{
    Department of Combinatorics and Optimization, University of Waterloo, Waterloo, Ontario N2L 3G1, Canada.
    Email: \texttt{lpostle@uwaterloo.ca}.
    Partially supported by NSERC under Discovery Grant No.\ 2019-04304, the Ontario Early Researcher Awards program and the Canada Research Chairs program.}
}
\date{}

\begin{document}

\renewcommand{\thefootnote}{\fnsymbol{footnote}} 
\renewcommand{\thefootnote}{\arabic{footnote}}

\maketitle

\begin{abstract}
  A graph $G$ is \textit{$k$-critical} if $\chi(G) = k$ and every proper subgraph of $G$ is $(k - 1)$-colorable, and if $L$ is a list assignment for $G$, then $G$ is \textit{$L$-critical} if $G$ is not $L$-colorable but every proper subgraph of $G$ is.
  In 2014, Kostochka and Yancey proved a lower bound on the average degree of an $n$-vertex $k$-critical graph tending to $k - \frac{2}{k - 1}$ for large $n$ that is tight for infinitely many values of $n$, and they asked how their bound may be improved for graphs not containing a large clique.
  Answering this question, we prove that there exists some $\varepsilon > 0$ for which the following holds.  If $k$ is sufficiently large and $G$ is a $K_{\omega + 1}$-free $L$-critical graph where $\omega \leq k - \log^{10}k$ and $L$ is a list assignment for $G$ such that $|L(v)| = k - 1$ for all $v\in V(G)$, then the average degree of $G$ is at least $(1 + \varepsilon)(k - 1) - \varepsilon \omega - 1$.  This result implies that for some $\varepsilon > 0$, for every graph $G$ satisfying $\omega(G) \leq \mathrm{mad}(G) - \log^{10}\mathrm{mad}(G)$ where $\omega(G)$ is the size of the largest clique in $G$ and $\mathrm{mad}(G)$ is the maximum average degree of $G$, the list-chromatic number of $G$ is at most $\left\lceil (1 - \varepsilon)(\mathrm{mad}(G) + 1) + \varepsilon\omega(G)\right\rceil$.
\end{abstract}

\section{Introduction}


A \textit{proper coloring} of a graph $G$ is an assignment of colors to the vertices of $G$ in which adjacent vertices receive different colors, and the \textit{chromatic number} of $G$, denoted $\chi(G)$, is the fewest number of colors needed to properly color $G$.  A graph is $k$-critical if it has chromatic number $k$ and every proper subgraph has chromatic number at most $k - 1$.  The study of critical graphs is fundamental to chromatic graph theory, and the study of their density is of particular interest.


\subsection{Density of critical graphs}

Let $f_{k}(n)$ denote the minimum average degree of a $k$-critical graph on $n$ vertices.  Dirac~\cite{D57} in 1957 first asked to determine $f_k(n)$, and Gallai~\cite{G63-1, G63-2} and Ore~\cite{O67} later reiterated the question.  The problem also appeared in Jensen and Toft's~\cite{JT95} book of two hundred graph coloring problems.  Trivially, $f_k(n) \geq k - 1$, and Brooks' Theorem~\cite{B41} implies that for $k\geq 4$ this inequality is strict, unless the graph is complete, that is unless $n = k$.  Following classical improvements by Dirac~\cite{D57} and Gallai~\cite{G63-1, G63-2} and further improvements by Krivelevich~\cite{K97} and Kostochka and Stiebitz~\cite{KS99} in the 1990s, Kostochka and Yancey \cite{KY14} essentially resolved the problem in 2014 by proving the bound $f_k(n) \geq k - \frac{2}{k - 1} - \frac{k^2 - 3k}{n(k - 1)}$, which is tight for every $n\equiv 1 \mod k - 1$ as shown by Ore~\cite{O67}.

Kostochka and Yancey~\cite{KY14} and Jensen and Toft~\cite{JT95} also asked to determine $f_k(n, \omega)$, the minimum average degree of a $k$-critical graph on $n$ vertices with no clique of size greater than $\omega$.  Kostochka and Stiebitz~\cite{KS00} proved that $f_k(n, r) \geq 2k - o(k)$ as $k\rightarrow\infty$ for any fixed $r$.  Krivelevich~\cite{K97} proved that for any $\alpha\in (0, 1)$, $f_k(n, \alpha k) \geq k - 1/(2 - \alpha) - o(1)$; however, this bound is not better than the bound of Kostochka and Yancey~\cite{KY14} on $f_k(n)$.  In~\cite{KP18}, we proved that for every $\alpha \in (0, 1/2)$, there exists $\varepsilon > 0$ such that $f_k(n, \alpha k) > (1 + \varepsilon)(k - 1)$. In this paper, we extend this result for $\alpha \in (1/2, 1)$, which is the substantially more difficult half of the interval.  In fact, we show that the dependence of $\varepsilon$ on $\alpha$ is linear, and our result provides the best known bound on $f_k(n, \omega)$ for $k/2 \leq \omega \leq k - \log^{10}k$, as follows.

\begin{restatable}{thm}{mainThm}
\label{main avg degree bound}
  There exists $\varepsilon > 0$ such that the following holds.  If $k$ is sufficiently large and $\omega \leq k - \log^{10}k$, then
  \begin{equation}\label{main bound equation}
    f_k(n, \omega) > (1 + \varepsilon)(k - 1) - \varepsilon\omega - 1.
  \end{equation}
\end{restatable}

We made no attempt to optimize the value of $\varepsilon$ in this paper for the sake of clarity in the proof, but there is room for improvement.  It is possible that Theorem~\ref{main avg degree bound} holds for $\varepsilon = 1$, although we do not expect to obtain a value of $\varepsilon$ close to that with our current methods.  We believe Theorem~\ref{main avg degree bound} actually holds for any $\omega \leq k$.

We actually derive Theorem~\ref{main avg degree bound} from a more general result, Theorem~\ref{nice local critical thm}, which can be thought of as a local strengthening of Theorem~\ref{main avg degree bound}.  Proving this stronger result is also crucial to the proof.  Before presenting our main result, we discuss an important connection to Reed's $\omega, \Delta, \chi$ Conjecture~\cite{R98}.

\subsection{$\omega, \chi$, and maximum average degree}

In 1998, Reed~\cite{R98} famously conjectured that every graph $G$ satisfies $\chi(G) \leq \left\lceil \frac{1}{2}(\Delta(G) + 1 + \omega(G))\right\rceil$, where $\Delta(G)$ is the maximum degree of a vertex in $G$ and $\omega(G)$ is the size of a largest clique.  Reed~\cite{R98} proved what we call an ``epsilon version'' of his conjecture: for some small $\varepsilon > 0$, every graph $G$ satisfies $\chi(G) \leq \lceil (1 - \varepsilon)(\Delta(G) + 1) + \varepsilon\omega(G)\rceil$.  Recently, Bonamy, Perrett, and the second author~\cite{BPP18} improved the value of $\varepsilon$ in this result to $1/26$ for graphs of sufficiently large maximum degree, Delcourt and the second author~\cite{DP17} improved this further to $1/13$, and Hurley, de Joannis de Verclos, and Kang~\cite{HdVK20} improved this to $1/8.5$.

In~\cite{KP18}, we conjectured that Reed's Conjecture can be strengthened by replacing $\Delta(G)$ with the \textit{maximum average degree} of $G$, denoted $\mad(G)$, which is $\max_{H\subseteq G}\ad(H)$ where $\ad(H)$ is the average degree of $H$, as follows.

\begin{conj}[Kelly and Postle~\cite{KP18}]\label{mad reeds conjecture}
    For every graph $G$,
  \begin{equation*}
    \chi(G) \leq \left\lceil\frac{1}{2}\left(\mad(G) + 1 + \omega(G)\right)\right\rceil.
  \end{equation*}
\end{conj}

Lower bounds on the density of critical graphs imply upper bounds on $\chi$ in terms of $\mad$.  Brooks' Theorem~\cite{B41} implies $\chi(G) < \mad(G) + 1$ unless $\omega(G) = \mad(G) + 1$ or $\mad(G) = 2$, and Kostochka and Yancey's~\cite{KY14} bound on $f_k(n)$ strengthens this bound on $\chi(G)$ to something tending to $\mad(G)$.  Kostochka and Stiebitz's~\cite{KS00} bound on $f_k(n, r)$ implies that if $G$ is $K_r$-free, then $\chi(G) \leq (1/2 + o(1))\mad(G)$, which confirms Conjecture~\ref{mad reeds conjecture} asymptotically in this case.
Using Theorem~\ref{main avg degree bound}, we derive an ``epsilon version'' of Conjecture~\ref{mad reeds conjecture}, as follows.

\begin{thm}\label{epsilon mad reeds}
  There exists $\varepsilon > 0$ such that the following holds.  For every graph $G$ such that $\omega(G) \leq \mad(G) - \log^{10}\mad(G)$,
  \begin{equation*}
    \chi(G) \leq \lceil(1 - \varepsilon)(\mad(G) + 1) + \varepsilon\omega(G)\rceil.
  \end{equation*}
\end{thm}

Theorem~\ref{epsilon mad reeds} extends our previous result~\cite[Theorem 1.15]{KP18} to the more difficult range $\mad(G)/2 \leq \omega(G) \leq \mad(G) - \log^{10}\mad(G)$ in the same way that Theorem~\ref{main avg degree bound} extends~\cite[Theorem 1.13]{KP18}.  Theorem~\ref{epsilon mad reeds} also strengthens Reed's~\cite{R98} bound for any graph satisfying $\omega(G) \leq \mad(G) - \log^{10}\mad(G)$.

\subsection{A local version: $\Gap$ and $\Save$}\label{intro local section}

We prove Theorem~\ref{main avg degree bound} in the more general setting of \textit{list coloring}.  A \textit{list assignment} for a graph $G$ is a collection of ``lists of colors'' $L = (L(v) \subset \mathbb N: v\in V(G))$, an \textit{$L$-coloring} of $G$ is a proper coloring $\phi$ such that $\phi(v) \in L(v)$ for every $v\in V(G)$, and $G$ is \textit{$L$-critical} if $G$ is not $L$-colorable but every proper subgraph is.

We also strengthen Theorem~\ref{main avg degree bound} by proving that it holds for graphs with $\sum_{v\in V(G)}\omega(v)/|V(G)| \leq \omega$, where $\omega(v)$ is the size of the largest clique containing $v$.

We actually strengthen this result even further.  To that end, we use the following notation.  If $G$ is a graph with list assignment $L$, then for each $v\in V(G)$, we let
  \begin{itemize}
  \item $\Gap_G(v) = d(v) + 1 - \omega(v)$ and
  \item $\Save_L(v) = d(v) + 1 - |L(v)|$.
  \end{itemize}
  We often omit the graph $G$ from the subscript of $\Gap$ if there is no ambiguity.  

The following is the main result in this paper.

\begin{restatable}{thm}{localThm}
\label{nice local critical thm}
  There exists $\eps > 0$ such that the following holds.  If $k$ is sufficiently large and $G$ is an $L$-critical graph for some list assignment $L$ satisfying $|L(v)| = k - 1$ and $\Gap(v) \geq \log^{10}k$ for each $v\in V(G)$, then
  \begin{equation}\label{nice bound equation}
    \sum_{v\in V(G)}\Save_L(v) > \sum_{v\in V(G)}\varepsilon \Gap(v).
  \end{equation}
\end{restatable}

Theorem~\ref{nice local critical thm} implies Theorems~\ref{main avg degree bound} and~\ref{epsilon mad reeds} fairly easily.  We include the proofs in Section~\ref{proof section}.

Besides being independently interesting, the stronger formulation of Theorem~\ref{nice local critical thm} is crucial to the proof.  Using $\Gap$ and $\Save$ allows us to shift to a new perspective from which we are better equipped to approach the problem.  In order to gain some intuition, we introduce the notion of the ``savings'' of a vertex with respect to a partial coloring and an ordering $\ordering$ of $V(G)$ (defined formally in Definition~\ref{random variable definitions}).  The savings is defined in such a way that if every vertex has savings larger than $\Save_L(v)$, then the partial coloring can be extended to an $L$-coloring of $G$ by coloring the uncolored vertices greedily in the ordering specified by $\ordering$.  Roughly, the savings for a vertex $v$ with respect to a partial coloring and an ordering $\ordering$ of $V(G)$ counts the total of the following:
\begin{itemize}
\item the number of neighbors of $v$ assigned a color not in $L(v)$,
\item the multiplicity less one of each color assigned to more than one of $v$'s neighbors, and
\item the number of uncolored neighbors $u$ of $v$ such that $u\ordering v$.
\end{itemize}

One typically approaches problems involving sparse graphs such as these by arguing that critical graphs do not have many vertices of low degree.  This approach is natural because vertices of low degree have small $\Save$ and thus require less savings.  Instead of focussing on degrees, we take a different approach by focussing on the ratio of $\Gap(v)$ and $\Save(v)$.
Very roughly, this ratio is important because we can find a partial coloring of a critical graph, using the probabilistic method, wherein although $\omega(v)$ colors may be forced in $N(v)$, the savings for $v$ is $\Omega(\Gap(v))$, as long as $v$ is not one of a few exceptional types.  By combining the probabilistic method with a unique application of the discharging method to handle these exceptional vertices, we can prove Theorem~\ref{nice local critical thm}.  We give a more detailed overview of the proof in Section~\ref{overview section}.

\section{Overview of the proof}\label{overview section}

Our strategy can be understood as an application of the method of discharging and reducible configurations, which is perhaps most well-known for being used to prove the Four Color Theorem~\cite{AH76, RSST97}.  To apply this method, one exhibits a list of configurations that are reducible, meaning they do not appear in a critical graph.  For the discharging, one assigns each vertex a \textit{charge} (simply a real number) such that the sum of the charges is some fixed negative constant and then provides ``discharging rules'' to obtain a new assignment of charges such that the sum is preserved.  The discharging rules are designed so that any vertex with negative final charge is in a reducible configuration.  Consequently, the graph is not critical, which can be used to obtain a proof by contradiction.

Our application of this method deviates from the traditional usage in several ways.
\begin{itemize}
\item The charge that we assign to each vertex is not a function of its degree as is common -- our charge assignment incorporates $\Gap$ and $\Save$.
\item We may have to apply our discharging argument arbitrarily many times, irrespective of $k$, whereas it typically suffices to apply discharging rules only once.
\item Our reducible configurations can be arbitrarily large and consequently our list of reducible configurations is not finite.  Moreover, we prove reducibility of some configurations using the probabilistic method.
\end{itemize}

\subsection{Discharging}\label{iterative discharging subsection}

Suppose $G$ and $L$ satisfy the hypotheses of Theorem~\ref{nice local critical thm} yet~\eqref{nice bound equation} does not hold.  In light of the discussion in Section~\ref{intro local section}, it would be natural to assign initial charges to the vertices of $G$ by giving each vertex $v$ a charge of $\Save_L(v) - \varepsilon\Gap(v)$ -- the total sum of the charges is at most zero, and vertices with small charge are in some sense easier to color.
If we are unable to find sufficient savings for a vertex $v$, then we ``discharge'' it, meaning we put it in a set $D$ and redistribute its charge to neighbors not in $D$ (this process is formalized in Lemma~\ref{inductive discharging lemma}).  In particular, $v$ sends $\varepsilon$ charge to each neighbor not in $D$, and we only discharge $v$ if the remaining charge of vertices in $D$ is positive.  Importantly, a vertex $v\in D$ always has $O(\Save_L(v))$ neighbors not in $D$, because otherwise we could put $v$ before these neighbors in the ordering $\ordering$ and have sufficient savings.  Since $D$ has positive charge, it is a proper subset of $V(G)$, so we can find a coloring of $G[D]$ as $G$ is critical.  If $L'$ is the list assignment for $G - D$ consisting of the remaining available colors, then a vertex $v\notin D$ has charge at least $\Save_{L'}(v) - \varepsilon\Gap_{G - D}(v)$.  Thus, the charges track the ratios of $\Gap$ and $\Save$ -- vertices in $D$ ``pay'' their neighbors for the fact that they no longer contribute to their $\Gap$.
We may have to iterate this discharging procedure arbitrarily many times, but since the discharged set retains positive charge at each iteration, this procedure would terminate.  We remark that in our formal proof, there is no need to keep track of these iterations; instead, it suffices to consider such a set $D$ of maximum cardinality (see the proof of Theorem~\ref{local critical thm}).   At this point, we are able to extend a precoloring of the discharged vertices to the remaining vertices using the probabilistic method.  In other words, all of the remaining vertices comprise a reducible configuration.  

\subsection{A technical result}

For technical reasons, we introduce some additional terms into the actual assignment of initial charges, and we prove a stronger, more technical result.  In order for our application of the probabilistic method to succeed, we need each vertex $v$ to have savings at least $\Omega(\log^{10}k)$.  We can achieve this if $\Gap(v) = \Omega(\log^{10}k)$ as assumed in Theorem~\ref{nice local critical thm}.  Thus, we need to take extra care if $\Gap_{G - D}(v)= o(\log^{10}k)$.  We also need to take extra care if $v$ is a vertex with $|L'(v)| = o(k)$, where $L'$ is the list assignment for $G - D$ consisting of the remaining available colors.  If $uv\in E(\overline{G[N(w)]})$ such that $L(u)\cap L(v) = \varnothing$ and $L(u)\cup L(v)\subseteq L(w)$, then we cannot use $u$ and $v$ to contribute to the savings for $w$.  However, if $|L(u)|, |L(v)| > k/2$, then we cannot have this situation.

These issues only arise if the vertex $v$ has many neighbors in $D$.  If we sacrifice a constant factor in the value of $\varepsilon$, we have a clever way to handle these issues.  When vertices are discharged, we instead have them send $9\varepsilon$ charge to each neighbor $v\notin D$: an extra $\varepsilon$ charge ``pays'' for the potential deficit in $\Gap(v) - \log^{10}k$ and $7\varepsilon$ charge ``pays'' for $k - |L(v)|$.  If a vertex is problematic, then it has stored ``extra charge'', and we can discharge it in the subsequent iteration.

The following is the main technical result in this paper.

\begin{restatable}{thm}{techThm}
\label{local critical thm}
  There exists $\eps > 0$ such that the following holds.  If $k$ is sufficiently large and $G$ is an $L$-critical graph for some list assignment $L$ such that for each vertex $v\in V(G)$ we have $|L(v)| \leq k$, then
  \begin{equation}\label{main thm equation}
    \sum_{v\in V(G)}(\Save_L(v) + \varepsilon\,\log^{10}k) > \sum_{v\in V(G)} (2\varepsilon\,\Gap(v) - 7\varepsilon(k - |L(v)|)).
  \end{equation}
\end{restatable}

Theorem~\ref{local critical thm} is easily seen to imply Theorem~\ref{nice local critical thm} -- we include a proof in Section~\ref{proof section}.  The remainder of the paper is devoted to the proof of Theorem~\ref{local critical thm}.  To that end, we introduce constants $\delta, \mu, \eta \in (0, 1)$ ($\delta$ is first used in Definition~\ref{nbrhood partition definition}, $\mu$ is first used in Definition~\ref{ways to save definition}, and $\eta$ is first used in Definition~\ref{heavy definition}), which we do not determine precisely, but we assume satisfy the following hierarchy for the remainder of the paper:
\begin{equation}\label{hierarchy}
  0 < 1/k_0 < \eps \ll \eta \ll \mu \ll \delta \ll 1.
\end{equation}
We assume $k > k_0$, and for each vertex $v\in V(G)$, we assign an initial charge of
\begin{equation*}
  \ch(v) = \Save_L(v) + \varepsilon\log^{10}k - 2\varepsilon\Gap(v)  + 7\varepsilon(k - |L(v)|).
\end{equation*}

\subsection{Reducible configurations}

In the proof of Theorem~\ref{local critical thm}, there are two types of reducible configurations, which are called \textit{saved} and \textit{dense}.  

\subsubsection*{Saved graphs.}

In Section~\ref{savings section} we define \textit{saved graphs} (with respect to $L$) formally, and we prove their reducibility using the probabilistic method.  To prove reducibility we build upon ideas from~\cite{KP18}, and thankfully we can avoid some technicalities involving concentration of random variables by invoking~\cite[Theorem~3.11]{KP18}.  
In order to show that saved graphs are reducible, we find a partial coloring in which each vertex $v$ has savings at least $\Save(v)$ using the probabilistic method.  We sample a random partial coloring and show that the savings for each vertex is large in expectation in any of the following situations (defined formally in Definition~\ref{ways to save definition} -- ``sufficiently many'' roughly means at least $\Save_L(v) / \mu$):
\begin{itemize}
\item sufficiently many neighbors of $v$ have many colors in their list that are not in $L(v)$, in which case we say $v$ is \textit{aberrant} (or \textit{slightly aberrant}),
\item sufficiently many pairs of non-adjacent neighbors of $v$ have lists of colors of size close to $|L(v)|$ (roughly $(1 \pm \delta)|L(v)|$, and in particular greater than $|L(v)|/2$), in which case we say $v$ is \textit{egalitarian-sparse}, 
\item the neighborhood contains the complement of a large bipartite subgraph with partition $(A, B)$ called a \textit{half-egalitarian bipartition} (see Definition~\ref{half-egalitarian bipartition definition}) with sufficiently many vertices in $A$, where vertices in $B$ have lists of colors of size close to $|L(v)|$ and vertices in $A$ have many non-neighbors in $B$ (at least $\delta d(v) / 20$), in which case we say $v$ is \textit{bipartite-sparse}, or
\item sufficiently many neighbors $u$ of $v$ satisfy $v\ordering u$ for some fixed ordering $\ordering$ of $V(G)$, in which case we say $v$ is \textit{$\ordering$-prioritized}.
\end{itemize}

If every vertex is either aberrant, slightly aberrant, egalitarian-sparse, bipartite-sparse, or $\ordering$-prioritized, then we say the graph is ``saved'' (see Definition~\ref{saved definition}).  An induced subgraph is saved if after precoloring the complement, it is saved with respect to the list assignment consisting of the remaining available colors.  Section~\ref{savings section} is devoted to proving the reducibility of saved subgraphs.

\subsubsection*{Dense graphs.}%
In Section~\ref{discharging structure section} we show (see Lemma~\ref{normal vertex structure lemma}) that if a vertex $v$ is not one of the following types (defined formally in Definition~\ref{heavy definition}), then it contains a reducible configuration in its neighborhood which we call a \textit{dense graph}.  That is, either 
\begin{itemize}
\item $v$ is aberrant, slightly aberrant, egalitarian-sparse, or bipartite sparse,
\item many neighbors of $v$ have much fewer than $|L(v)|$ available colors (at most $(1 - \delta)|L(v)|$) and moreover $\Gap(v)$ is large (at least $3\delta d(v) / 4$), in which case we say $v$ is \textit{very lordly},
\item $v$ has high charge (at least $\eta \Gap(v)$), in which case we say $v$ is \textit{heavy}, 
\item $v$ has many heavy neighbors (at least $d(v) / 2$) and moreover $\Gap(v)$ is small (at most $\delta d(v) / 4$), in which case we say $v$ is \textit{sponsored}, or
\item there is an induced subgraph $H\subseteq G[N(v)]$ and a matching $M$ in $\overline H$ such that
\begin{equation*}
  |E(\overline H)| < |M|(|V(H)| - |M|) - \sum_{u\in V(H)}\Save_L(u).
\end{equation*}
\end{itemize}
In the latter case, we say that $H$ is \textit{dense with respect to $L$}.  The reducibility of dense graphs follows from~\cite[Theorem 4.1]{KP18}.


Note that for any graph $H$, a maximum matching in $\overline H$ has size at least $(|V(H)| - \omega(H))/2$, and thus $\overline{G[N(v)]}$ has a matching of size at least $\Gap(v)/2$.  If we let $H$ be the subgraph induced by neighbors of $v$ with roughly $L(v)$ available colors, then if $H$ is not dense with respect to $L$, either on the order of $\Gap(v)d(v)$ pairs of $v$'s neighbors in $H$ are non-adjacent, or $|N(v)\setminus V(H)| \geq \Gap(v)/2$,  or $v$ has many neighbors $u$ for which $\Save_L(u)$ is comparatively large.  If $v$ is not heavy, then in the former case, we show $v$ is egalitarian-sparse, in the second case we show $v$ is aberrant, bipartite-sparse or very lordly, and in the latter case we show that $v$ is either aberrant, slightly aberrant, or sponsored.

We can now describe how we choose the ``discharged'' set $D$ mentioned in Section~\ref{iterative discharging subsection}.  We let $S_0$ be the vertices of $G$ that are aberrant, slightly aberrant, egalitarian-sparse, or bipartite-sparse.  For $i\geq 1$, we let $S_i$ be the vertices not in $\cup_{j=0}^{i-1}S_j$ with enough neighbors in $\cup_{j=0}^{i-1}S_j$ so that for any ordering $\ordering$ of $V(G)$ satisfying $u\ordering v$ where $u\in S_i$ and $v\in S_j$ for $i > j$, every vertex in $\cup_{i\geq 1}S_i$ is $\ordering$-prioritized.  If $\cup_iS_i = V(G)$, then by construction, the graph $G$ is saved.  If not, we let $\mathcal L$ be the very lordly vertices not in $\cup_i S_i$, and we let $D = V(G)\setminus (\mathcal L \cup \bigcup_i S_i)$.  If $D = \varnothing$, then $\mathcal L = \varnothing$, since otherwise the vertex in $\mathcal L$ with the fewest available colors has enough neighbors in $\cup_i S_i$ to be $\ordering$-prioritized.  Lemma~\ref{normal vertex structure lemma} implies that the vertices in $D$ are either heavy or sponsored, which in turn implies that the total charge of vertices in $D$ is large.  Since vertices in $D$ do not have many neighbors in $\cup S_i$, they can afford to send $9\varepsilon$ charge to each neighbor there.  Lemma~\ref{in big clique with many very lordly nbrs is aberrant} implies that heavy and sponsored vertices with many very lordly neighbors are aberrant.  Thus, the vertices in $D$ do not have many very lordly neighbors and can afford to send them each $9\varepsilon$ charge as well.  

\section{Coloring a saved graph with the local naive random coloring procedure}\label{savings section}

The main result of this section is Theorem~\ref{saved is colorable}, which implies that saved subgraphs are reducible.  First, we need several definitions.

\subsection{The local naive random coloring procedure}

In the proof of Theorem~\ref{saved is colorable}, we analyze a random coloring procedure introduced in~\cite[Section 3]{KP18}  called the ``local naive random coloring procedure.''  In this subsection, we recall the details of this procedure.  For additional background, we refer the reader to~\cite{KP-arxiv}, which corrects some errors appearing in~\cite{KP18} (see also~\cite{KP-corrigendum}).    

We let $G$ be a graph with list assignment $L$, and we let $(L, M)$ be a correspondence assignment (first defined in~\cite{DP15}, see also~\cite{KP18}) for $G$.  It is possible to prove our results without the notion of correspondence coloring, but correspondence coloring simplifies some calculations in Section~\ref{expectation section}.   A \textit{naive partial $(L, M)$-coloring} of $G$ is a pair $(\phi, U)$ where $\phi : V(G)\rightarrow\mathbb N$ such that $\phi(u)\in L(u)$ for every $u\in V(G)$ and $U\subseteq V(G)$ is a set of \textit{uncolored vertices} such that $\phi|_{V(G)\setminus U}$ is an $(L, M)$-coloring of $G - U$.  A vertex is called \textit{colored} if it is not in $U$.  We sample a naive partial $(L, M)$-coloring randomly using the following procedure.

\begin{define}\label{random coloring procedure}
  Let $\varepsilon \in (0, 1)$ and $\rho \in [0, 1]$, and let $K_{\varepsilon, \rho} = 999\rho e^{-\rho/(1 - \varepsilon)} / 1000$.
  If for each $v\in V(G)$, $|L(v)| \geq (1 - \varepsilon)d(v)$ and $G$ has minimum degree at least $1000 / (1 - \eps)^2$, then the \textit{local naive random coloring procedure with activation probability $\rho$ and $\varepsilon$-equalizing coin-flips} samples a random naive partial $(L, M)$-coloring $(\phi, \uncolvtcs)$ (and a set of activated vertices $A$) in the following way.  For each $v \in V(G)$,
  \begin{enumerate}
  \item let $v\in A$ independently at random with probability $\rho$,
  \item choose $\phi(v)\in L(v)$ independently and uniformly at random, 
  \item\label{uncoloring-rule} let $v\in U'$ if there exists $u\in N(v)\cap A$ such that $|L(u)| \geq |L(v)|$ and $\phi(u)\phi(v) \in M_{uv}$,
  \item\label{equalizing-coin-flip-step} conduct an ``$\eps$-equalizing coin flip'' for $v$ and each $c \in L(v)$ that is ``heads'' with probability $1 - \rho^{-1}K_{\varepsilon, \rho}/\ProbCond{v\notin \uncolvtcs'}{\phi(v) = c}$, and
  \item let $v \in U$ if one of the following holds: $v \notin A$, $v \in U'$, or the coin flip for $v$ and $\phi(v)$ is heads.
  \end{enumerate}
\end{define}
We require that every $v \in V(G)$ satisfies $|L(v)| \geq (1 - \eps)d(v)$ and that $G$ has minimum degree at least $1000 / (1 - \eps)^2$ so that Step~\ref{equalizing-coin-flip-step} of Definition~\ref{random coloring procedure} is well-defined.  Indeed,~\cite[Proposition~3.7]{KP-corrigendum} implies that $1 - \rho^{-1}K_{\varepsilon, \rho} / \ProbCond{v \notin U'}{\phi(v) = c} \in [0, 1]$ for every $v\in V(G)$ and $c\in L(v)$, and as proved in~\cite[Proposition~3.9]{KP-corrigendum}, this step ensures that $\ProbCond{v \notin U}{\phi(v) = c} = K_{\varepsilon, \rho}$.

\begin{define}
  Let $\sigma\in [0, 1)$.  For each $v\in V(G)$ and $u\in N(v)$, we say $u$ is a \textit{$\sigma$-egalitarian neighbor} of $v$ if $u$ has at least $(1 - \sigma)|L(v)|$ available colors.  We let $\egal_\sigma(v)$ denote the set of $\sigma$-egalitarian neighbors of $v$.
\end{define}

\begin{define}\label{random variable definitions}
  For each $v\in V(G)$, we define the following random variables, with respect to $\sigma \in [0, 1)$ and an ordering $\ordering$ of $V(G)$, where $(\phi, U)$ is a random naive partial $(L, M)$-coloring of $G$ and $A \supseteq V(G)\setminus U$ is the set of activated vertices.
  \begin{itemize}
  \item Let $\aberrance_{v,\sigma}(\phi, U)$ count the number of colored $\sigma$-egalitarian neighbors $u$ of $v$ such that $\phi(u)$ is not matched by $M_{uv}$.
  \item For each $X\subseteq N(v)$, let $\pairs_{v, X}(\phi, U)$ and $\trips_{v, X}(\phi, U)$ count the number of pairs and triples respectively of colored neighbors of $v$ in $X$ that receive colors that are matched to the same color in $L(v)$.
  \item Let $\unact_{v, \ordering}(A)$ count the number of non-activated neighbors $u$ of $v$ such that $u\ordering v$.
  \end{itemize}      
  More precisely, we have that
  \begin{align*}
    \aberrance_{v,\sigma}(\phi, U) &= |\{u\in \egal_\sigma(v) \setminus U  : \phi(u)\notin V(M_{uv})\}|,\\
    \pairs_{v, X}(\phi, U)  &= |\{x,y\in X\setminus U, c\in L(v) : \phi(x)c\in M_{xv}\text{ and }\phi(y)c\in M_{yv}\}|,\\
    \trips_{v, X}(\phi, U) &= |\{x,y,z\in X\setminus U, c\in L(v) :  \phi(w)c\in M_{wv}~\forall~w\in\{x, y, z\}\}|, \text{ and}\\
    \unact_{v, \ordering}(A) &= |\{u\in N(v)\setminus A : u \ordering v\}|.
  \end{align*}
  We also define the following:
  \begin{itemize}
  \item $\savings'_{v, \sigma, \ordering}(\phi, U, A) = \aberrance_{v, \sigma}(\phi, U)  + \pairs_{v, \egal_\sigma(v)}(\phi, U) -\allowbreak \trips_{v, \egal_\sigma(v)}(\phi, U) +\allowbreak \unact_{v, \ordering}(A)$, and
  \item $\savings_{v, \sigma, \ordering}(\phi, U, A) = \aberrance_{v, \sigma}(\phi, U)  + \max\left\{\pairs_{v, X}(\phi, U) - \trips_{v, X}(\phi, U) :\allowbreak X \subseteq \egal_\sigma(v)\right\} \allowbreak+ \unact_{v, \ordering}(A)$.
  \end{itemize}
\end{define}

As alluded to in Section~\ref{intro local section}, if $\sigma \in [0, 1]$, $\ordering$ is an ordering of $V(G)$, and $(\phi, U)$ is a naive partial $(L, M)$-coloring of $G$ satisfying $\savings_{v, \sigma, \ordering}(\phi, U, V(G)\setminus U) > \Save_L(v)$ for every $v \in V(G)$, then $\phi$ can be extended to an $(L, M)$-coloring of $G$ by coloring $V(G)\setminus U$ greedily in the ordering specified by $\ordering$.
In~\cite[Theorem 3.11]{KP18}, we proved roughly that under some technical conditions, in order to find such a naive partial $(L, M)$-coloring, it suffices to show that under the local naive random coloring procedure, the expected value of $\savings'_{v, \sigma, \ordering}$ is larger than $\Save_L(v)$ for each vertex $v$.  In this paper, we need the following slight strengthening of this result in which $\savings'$ is replaced with $\savings$ (note that $\savings'_{v, \sigma, \ordering}(\phi, U, A) \geq \savings_{v, \sigma, \ordering}(\phi, U, A)$).

\begin{thm}
  \label{local sparsity lemma}
  For every $\xi_1, \xi_2 > 0$, $\varepsilon,\sigma\in[0, 1)$, and $\rho \in [0, 1]$, there exists $\Delta_0$ such that the following holds.  Let $\Delta \geq \Delta_0$, and let $G$ be a graph with correspondence assignment $(L, M)$ such that
  \begin{enumerate}
  \item $G$ has maximum degree at most $\Delta$ and minimum degree at least $1000/(1 - \varepsilon)^2$, and
  \item $\Delta \geq |L(v)| \geq (1 - \varepsilon)d(v)$ for each $v \in V(G)$.
  \end{enumerate}
  If $\ordering$ is an ordering of $V(G)$ such that $\Expect{\savings_{v, \sigma, \ordering}} \geq \max\{(1 + \xi_1)\Save_L(v), \xi_2\log^{\logexp}\Delta\}$ for all $v \in V(G)$ for the local naive random coloring procedure with activation probability $\rho$ and $\eps$-equalizing coin-flips,
  then $G$ is $(L, M)$-colorable.
\end{thm}

The proof of Theorem~\ref{local sparsity lemma} can be obtained by a straightforward modification of the proof of~\cite[Theorem 3.11]{KP18}, so we omit the details.
The majority of the remainder of this section is devoted to proving that in a ``saved graph'' (see Definition~\ref{saved definition}), the expected value of  $\savings_{v, \sigma, \ordering}$ is large for each vertex $v$.  

\subsection{Ways to save}

We need to partition the neighborhood of each vertex according to the size of each neighbor's list of colors, as follows.  Recall $\delta \in (0, 1)$ satisfies~\eqref{hierarchy}.
\begin{define}\label{nbrhood partition definition}
  Let $G$ be a graph with list assignment $L$, let $v\in V(G)$, and let $u\in N(v)$.
  \begin{itemize}
  \item If $|L(u)| < (1 - \delta)|L(v)|$, then we say $u$ is a \textit{subservient neighbor} of $v$.
  \item If $|L(u)| \in [(1 - \delta)|L(v)|, (1 + \delta / 3)|L(v)|)$, then we say $u$ is an \textit{egalitarian neighbor} of $v$.
  \item If $|L(u)| \geq (1 + \delta / 3)|L(v)|$, then we say $u$ is a \textit{lordlier neighbor} of $v$.
  \item If $|L(u)| \geq |L(v)| + \delta\Gap(v) / 3$, then we say $u$ is a \textit{slightly lordlier neighbor} of $v$.
  \end{itemize}
  For convenience, we will let $\slaberrant(v)$ denote the set of slightly lordlier neighbors of $v$, $\aberrant(v)$ denote the set of lordlier neighbors of $v$, $\egal(v)$ denote the set of egalitarian neighbors of $v$, and $\subserv(v)$ denote the set of subservient neighbors of $v$.  
\end{define}

Recall that $\egal_\sigma(v)$ is the set of $\sigma$-egalitarian neighbors of $v$, which are neighbors $u$ of $v$ with at least $(1 - \sigma)|L(v)|$ available colors, and note that $\egal(v) \subseteq \egal_{2/3}(v)$.  We apply Theorem~\ref{local sparsity lemma} with $2/3$ as $\sigma$, $99/100$ as $\rho$, and $11\eps$ as $\eps$, so in this subsection we let $K = K_{11\eps, 99/100} = 999\cdot(99/100) e^{-(99/100) / (1 - 11\eps)} / 1000$.  Since $\eps \ll 1$, we have $K \geq .367$.  Recall also that $\mu \in (0, 1)$ satisfies~\eqref{hierarchy}.
The following definitions provide sufficient conditions for $\savings_{v, 2/3, \ordering}$ to be sufficiently large in expectation, for a random naive partial coloring sampled with the local naive random coloring procedure with activation probability $99/100$ and $11\eps$-equalizing coin flips.  

\begin{define}\label{half-egalitarian bipartition definition}
  A pair $(A, B)$ of disjoint subsets of $N(v)$ is a \textit{half-egalitarian bipartition} for $v$ if
  \begin{itemize}
  \item $B\subseteq \egal(v)$,
  \item $A\subseteq \egal_{2/3}(v)\cap\subserv(v)$, and
  \item each vertex $u\in A$ has at least $\delta d(v)/ 20$ non-neighbors in $B$.
  \end{itemize}
\end{define}

\begin{define}\label{ways to save definition}
  Let $G$ be a graph with list assignment $L$.
  We say a vertex $v\in V(G)$ is
  \begin{itemize}
  \item \textit{aberrant} with respect to $L$ and $k$ if
    \begin{equation*}
      |\aberrant(v)| \geq \left(\Save_L(v) + \varepsilon\log^{10}k\right)/\mu,
    \end{equation*}
  \item \textit{slightly aberrant} with respect to $L$ and $k$ if
     \begin{equation*}
      |\slaberrant(v)| \geq \frac{d(v)}{\Gap(v)}\left(\Save_L(v) + \varepsilon\log^{10}k\right)/\mu,
    \end{equation*}
  \item \textit{egalitarian-sparse} with respect to $L$ and $k$ if
    \begin{equation*}
      |E(\overline{G[\egal(v)]})| \geq d(v)\left(\Save_L(v) + \varepsilon\log^{10}k\right) / \mu,
    \end{equation*}
  \item \textit{bipartite-sparse} with respect to $L$ and $k$ if $v$ has a half-egalitarian bipartition $(A, B)$ such that
    \begin{equation*}
      |A| \geq 100\left(\Save_L(v) + \varepsilon\log^{10}k\right) / \mu,
    \end{equation*}
  \item and \textit{$\ordering$-prioritized} with respect to $L$ and $k$ if $\ordering$ is an ordering of $V(G)$ such that $v$ has at least $(\Save_L(v) + \varepsilon\log^{10}k)/\mu$ neighbors $u$ such that $v\ordering u$.
  \end{itemize}
\end{define}

As Lemma~\ref{aberrance savings} shows, an aberrant or slightly aberrant vertex $v$ has large expected savings because $\aberrance_v$ is large in expectation.  Each lordlier or slightly lordlier neighbor of $v$ has a good chance to receive a color not matched to a color in $L(v)$.

As Lemma~\ref{egal-sparse savings} shows, an egalitarian-sparse vertex $v$ has large expected savings because $\pairs_{v, \egal(v)} - \trips_{v, \egal(v)}$ is large in expectation.  Every pair of non-adjacent egalitarian neighbors of $v$ has a good chance to receive the same color.  Here it is important to consider correspondence coloring, rather than list coloring.  Correspondence coloring allows us to essentially assume that two neighboring vertices' lists of colors have as many colors in common as possible.  Since $\delta < 1/2$, two egalitarian neighbors of $v$ are forced to have some colors that correspond to the same color in $L(v)$.

Lemma~\ref{bipartite-sparse savings} shows that a bipartite-sparse vertex $v$ also has large expected savings because if $(A, B)$ is a half-egalitarian bipartition for $v$, then $\pairs_{v, A \cup B} - \trips_{v, A\cup B}$ is large in expectation.  Indeed, each vertex in $A$ has a good chance to receive the same color as one of its non-neighbors in $B$.  Here we also use correspondence coloring to force a vertex in $A$ and a vertex in $B$ to have some colors that correspond to the same color in $L(v)$, since $2 - \delta - (2/3) > 1$.

\subsection{Expectations}
\label{expectation section}
In this subsection, we let $G$ be a graph with list assignment $L$, and we let $(L, M)$ be a correspondence assignment for $G$.  We assume $(L, M)$ is \textit{total}, meaning for each $uv\in E(G)$, the matching $M_{uv}$ saturates at least one of $\{u\}\times L(u)$ or $\{v\}\times L(v)$.

We prove a series of lemmas that show that $\savings_{v, 2/3, \ordering}$ is sufficiently large if a vertex $v$ satisfies one of the properties defined in Definition~\ref{ways to save definition}.  To that end, we let $(\phi, U)$ be randomly sampled according to the local naive random coloring procedure with activation probability $99/100$ and $11\varepsilon$-equalizing coin-flips, and we let $A$ denote the set of activated vertices and $U' \subseteq U$ as in Definition~\ref{random coloring procedure}.

The first such lemma will be applied to vertices that are either aberrant or slightly aberrant.

\begin{lemma}\label{aberrance savings}
  Let $0 < \mu \ll \delta \ll 1$.  For each $v\in V(G)$ such that $|L(v)| \leq d(v)$,
  \begin{equation*}
    \Expect{\aberrance_v} \geq 2\mu\cdot\max\left\{|\aberrant(v)|, \frac{\Gap(v)}{d(v)}|\slaberrant(v)|\right\}.
  \end{equation*}
\end{lemma}

\begin{proof}
  Let
  \begin{equation*}
    \totesabs_v(\phi) = |\{u \in N(v) : \phi(u)\notin V(M_{uv})\}|,
  \end{equation*}
  and note that $\Expect{\aberrance_v} = K\cdot\Expect{\totesabs_v}$.
  For each $u\in \aberrant(v)$,
  \begin{equation*}
    \Prob{\phi(u)\notin V(M_{uv})} \geq \frac{\delta/3}{1 + \delta/3},
  \end{equation*}
  and for each $u\in \slaberrant(v)$, since $d(v) \geq |L(v)|$ and $d(v) \geq \Gap(v)$,
  \begin{equation*}
    \Prob{\phi(u)\notin V(M_{uv})} \geq \frac{\delta\Gap(v)/3}{|L(v)| + \delta\Gap(v)/3} \geq \frac{\delta}{3 + \delta}\frac{\Gap(v)}{d(v)}.
  \end{equation*}
  Therefore it follows that
  \begin{equation*}
    \Expect{\totesabs_v} \geq \frac{\delta}{3 + \delta}\max\left\{|\aberrant(v)|, \frac{\Gap(v)}{d(v)}|\slaberrant(v)|\right\}.
  \end{equation*}
  Since $\Expect{\aberrance_v} = K\cdot\Expect{\totesabs_v}$ and $\mu \ll \delta$, the result follows.  
\end{proof}

The next two lemmas will be applied to the vertices that are egalitarian-sparse and bipartite-sparse, respectively.  First, we prove the following proposition which will be used in the proof of both of these lemmas.

\begin{prop}\label{prop:keep-squared}
  Let $0 < 1 / d_0 \ll \eps \ll 1$.  If every $v \in V(G)$ satisfies $d(v) \geq |L(v)| \geq (1 - 11\eps)d(v) \geq d_0$, then for every pair of non-adjacent vertices $x, y \in V(G)$ and for every pair of colors $c_x \in L(x)$ and $c_y \in L(y)$, we have
  \begin{equation*}
    \ProbCond{x, y \notin U}{\phi(x) = c_x,~\phi(y) = c_y} \geq \frac{99}{100} K^2.
  \end{equation*}  
\end{prop}
\begin{proof}
  Recall that $\rho = 99/100$ and that $U' \subseteq U$ is the set from Step~\ref{uncoloring-rule} of the local naive random coloring procedure with activation probability $\rho$ and $11\eps$-equalizing coin flips.  We assume without loss of generality that $d(y) \geq d(x)$.  Let $S_x = \{u \in \egal_0(x) : c_xc \in M_{xu}~\text{for some}~c \in L(u)\}$, and define $S_y$ analogously.  Let $B_1 = \{u \in S_x \cap S_y: c_xc \in M_{xu}~\text{and}~c_yc \in M_{yu}~\text{for some}~c \in L(u)\}$, and let $B_2 = (S_x \cap S_y)\setminus B_1$.  By the definition of our random coloring procedure, we have
  \begin{equation*}
    \ProbCond{x, y \notin U}{\phi(x) = c_x,~\phi(y) = c_y} = K^2\frac{\ProbCond{x, y \notin U'}{\phi(x) = c_x,~\phi(y) = c_y}}{\ProbCond{x \notin U'}{\phi(x) = c_x}\cdot \ProbCond{y \notin U'}{\phi(y) = c_y}}.
  \end{equation*}
  Moreover,
  \begin{equation*}
    \ProbCond{x, y \notin U'}{\phi(x) = c_x,~\phi(y) = c_y} =
    \prod_{u \in (S_x \cup S_y)\setminus B_2}\left(1 - \frac{\rho}{|L(u)|}\right)\prod_{u \in B_2}\left(1 - \frac{2\rho}{|L(u)|}\right), 
  \end{equation*}
  and
  \begin{align*}
    \ProbCond{x \notin U'}{\phi(x) = c_x} &= \prod_{u \in S_x}\left(1 - \frac{\rho}{|L(u)|}\right) && \mathrm{and}\\
    \ProbCond{y \notin U'}{\phi(y) = c_y} &= \prod_{u \in S_y}\left(1 - \frac{\rho}{|L(u)|}\right).
  \end{align*}
  Combining the above, we obtain
  \begin{align*}
    \ProbCond{x, y \notin U}{\phi(x) = c_x,~\phi(y) = c_y} &= K^2 \left.\prod_{u \in B_2}\frac{1 - 2\rho / |L(u)|}{(1 - \rho / |L(u)|)^2}\middle / \prod_{u \in B_1}\left(1 - \frac{\rho}{|L(u)|}\right)\right.\\
    &\geq K^2 \prod_{u \in B_2}\frac{1 - 2\rho / |L(u)|}{(1 - \rho / |L(u)|)^2}.
  \end{align*}
  For every $u \in B_2$, since $|L(u)| \geq d_0$, we have $\frac{1 - 2\rho / |L(u)|}{(1 - \rho / |L(u)|)^2} \geq 1 - \frac{2\rho^2}{|L(u)|^2} \geq 1 - \frac{3\rho^2}{d(x)^2}$.  Thus, since $d(x) \geq d_0$, we have
  \begin{equation*}
    \prod_{u \in B_2}\frac{1 - 2\rho / |L(u)|}{(1 - \rho / |L(u)|)^2} \geq \left(1 - \frac{3\rho^2}{d(x)^2}\right)^{d(x)} \geq \frac{99}{100}.
  \end{equation*}
  The result follows by combining the above two inequalities.
\end{proof}

\begin{lemma}\label{egal-sparse savings}
  Let $0 < 1 / d_0 \ll \eps \ll \mu \ll \delta \ll 1$.  
  If every $v \in V(G)$ satisfies $d(v) \geq |L(v)| \geq (1 - 11\varepsilon)d(v) \geq d_0$, then for every $v\in V(G)$, we have
  \begin{equation*}
    \Expect{\pairs_{v, \egal(v)} - \trips_{v, \egal(v)}} \geq 2\mu\frac{|E(\overline {G[\egal(v)]})|}{d(v)}.
  \end{equation*}
\end{lemma}

\begin{proof}
  We let $T(H)$ denote the set of triangles in a graph $H$.
  We define the following random variables for each $c\in L(v)$:
  \begin{align*}
    &\totespairs_{v,c}(\phi) = |\{x, y \in \egal(v): \phi(x)c\in M_{xv}\text{ and }\phi(y)c\in M_{yv}\}|, \text{\ and}\\
    &\totestrips_{v,c}(\phi) = |\{x, y, z\in \egal(v): \phi(x)c\in M_{xv}, \phi(y)c\in M_{yv},\text{ and }\phi(z)c\in M_{zv}\}|,
  \end{align*}
  and we define $\totespairs_v(\phi) = \sum_{c\in L(v)}\totespairs_{v,c}(\phi)$ and $\totestrips_v(\phi) = \sum_{c\in L(v)}\totestrips_{v,c}(\phi)$.

  For each $c\in L(v)$, let $H_c$ be the subgraph of $G[\egal(v)]$ defined as follows.  A vertex $x\in \egal(v)$ is in $V(H_c)$ if there exists a color $c_x\in L(x)$ such that $cc_x \in M_{vx}$, and $xy \in E(H_c)$ if $xy\in E(G)$ and moreover $c_xc_y \in M_{xy}$, where $cc_x \in M_{vx}$ and $cc_y \in M_{vy}$.  For every pair $xy \in E(\overline H_c)$, by Proposition~\ref{prop:keep-squared}, we have $\ProbCond{x, y \notin U}{\phi(x) = c_x,~\phi(y) = c_y} \geq 99 K^2 / 100$.  Moreover, for every triple $xyz\in T(\overline H_c)$, we have $\ProbCond{x, y, z\notin U}{\phi(w) = c_w~\forall~w\in\{x,y,z\}} \leq \ProbCond {x \notin U}{\phi(x) = c_w} = K$.
  Hence,
  \begin{equation}\label{pairs-trips-lower-bound-by-total}
    \Expect{\pairs_{v,\egal(v)} - \trips_{v,\egal(v)}} \geq \frac{99}{100}K^2\Expect{\totespairs_v} - K\Expect{\totestrips_v}.
  \end{equation}

  For every $c\in L(v)$, since $\delta\ll 1$,
  \begin{equation}    \label{total pairs expectation}
    \Expect{\totespairs_{v,c}}  = \sum_{xy\in E(\overline {H_c})}\frac{1}{|L(x)||L(y)|} \geq \frac{|E(\overline{H_c})|}{(1 + \delta/3)^2|L(v)|^2} \geq \frac{|E(\overline{H_c})|}{(1 + \delta)|L(v)|^2}.
  \end{equation}
  Similarly,
  \begin{equation}\label{trips expectation with triangles}
    \Expect{\totestrips_{v,c}} = \sum_{xyz\in T(\overline{H_c})}\frac{1}{|L(x)||L(y)||L(z)|} \leq \frac{|T(\overline{H_c})|}{(1 - \delta)^3|L(v)|^3}.
  \end{equation}
  Rivin~\cite{R02} proved that every graph $H$ satisfies
  \begin{equation}
    \label{triangle-bound}
    |T(H)| \leq \frac{(2|E(H)|)^{\frac{3}{2}}}{6}.
  \end{equation}
    By \eqref{trips expectation with triangles} and \eqref{triangle-bound},
  \begin{equation}
    \label{total trips expectation}
    \Expect{\totestrips_{v,c}} \leq \frac{\sqrt{2}|E(\overline{H_c})|^{3/2}}{3(1 - \delta)^3|L(v)|^3}.
  \end{equation}
  It follows from \eqref{pairs-trips-lower-bound-by-total}, \eqref{total pairs expectation}, and \eqref{total trips expectation} that
  \begin{equation*}
    \Expect{\pairs_{v, \egal(v)} - \trips_{v, \egal(v)}} \geq \frac{K}{|L(v)|^2}\sum_{c\in L(v)}|E(\overline{H_c})|\left(\frac{99K}{100(1 + \delta)} - \frac{\sqrt{2|E(\overline{H_c})|}}{3(1 - \delta)^3|L(v)|}\right),
  \end{equation*}
  and since $|E(\overline{H_c})| \leq \binom{d(v)}{2}$ for each $c \in L(v)$, $|L(v)| \geq (1 - 11\eps)d(v)$, and $\eps \ll \delta \ll 1$,
  \begin{equation}\label{pairs-triples-bound-per-color}
    \Expect{\pairs_{v, \egal(v)} - \trips_{v, \egal(v)}} \geq \frac{K}{|L(v)|^2}\sum_{c\in L(v)}{|E(\overline{H_c})|}/{100}.
  \end{equation}
  
  For every $x, y \in \egal(v)$, let $C_{x, y} = \{c\in L(v) : x, y \in V(H_c)\}$.  We claim that $|C_{x, y}| \geq (1 - 2\delta)|L(v)|$ for each $x, y\in \egal(v)$.  Suppose $|L(y)| < |L(v)|$, or else $|C_{x,  y}| \geq (1 - \delta)|L(v)| \geq (1 - 2\delta)|L(v)|$ since $(L, M)$ is total, as claimed.  Now $|L(x)| + |L(y)| - |C_{x, y}| \leq |L(v)|$.  Hence, $|C_{x, y}| \geq 2(1 - \delta)|L(v)| - |L(v)| = (1 - 2\delta)|L(v)|$, as claimed.  Now
  \begin{equation}\label{color sparsity double count}
    \sum_{c\in L(v)}|E(\overline{H_c})| \geq \sum_{xy \in E(\overline{G[\egal(v)]})}|C_{x, y}| \geq (1 - 2\delta)|L(v)|\cdot|E(\overline{G[\egal(v)]})|.
  \end{equation}
  The result follows by combining~\eqref{pairs-triples-bound-per-color} and~\eqref{color sparsity double count}, since $\mu \ll \delta \ll 1$.
\end{proof}

The next lemma will be applied to the vertices that are bipartite-sparse.
\begin{lemma}\label{bipartite-sparse savings}
  Let $0 < 1 / d_0 \ll \eps \ll \mu \ll \delta \ll 1$.  If every $v\in V(G)$ satisfies $d(v) \geq |L(v)| \geq (1 - 11\eps)d(v) \geq d_0$, and 
  if $v$ has a half-egalitarian bipartition $(A, B)$ with $|A| \leq d(v) / 100$, then
  \begin{equation*}
    \Expect{\pairs_{v, A\cup B} - \trips_{v, A\cup B}} \geq 2\mu |A|.
  \end{equation*}
\end{lemma}
\begin{proof}
  Let $X = A \cup B$.
  We define the following random variables for each $c\in L(v)$:
  \begin{align*}
    &\totespairs_{v,c}(\phi) = |\{x, y \in X: \phi(x)c\in M_{xv}\text{ and }\phi(y)c\in M_{yv}\}|, \text{\ and}\\
    &\totestrips_{v,c}(\phi) = |\{x, y, z\in X: \phi(x)c\in M_{xv}, \phi(y)c\in M_{yv},\text{ and }\phi(z)c\in M_{zv}\}|,
  \end{align*}
  and we define $\totespairs_v(\phi) = \sum_{c\in L(v)}\totespairs_{v,c}(\phi)$ and $\totestrips_v(\phi) = \sum_{c\in L(v)}\totestrips_{v,c}(\phi)$.

  For each $c\in L(v)$, let $H_c$ be the subgraph of $G[X]$ defined as follows.  A vertex $x\in X$ is in $V(H_c)$ if there exists a color $c_x\in L(x)$ such that $cc_x \in M_{vx}$, and $xy \in E(H_c)$ if $xy\in E(G)$ and moreover $c_xc_y \in M_{xy}$, where $cc_x \in M_{vx}$ and $cc_y \in M_{vy}$.  As in Lemma~\ref{egal-sparse savings}, we have  
  \begin{equation*}
    \Expect{\pairs_{v, X} - \trips_{v,X}} \geq \frac{99}{100}K^2\Expect{\totespairs_v} - K\Expect{\totestrips_v}.
  \end{equation*}
  For every $c\in L(v)$,
  
  \begin{minipage}{.4\textwidth}
    \centering
  \begin{equation*}
    \Expect{\totespairs_{v,c}} = \sum_{xy\in E(\overline {H_c})}\frac{1}{|L(x)||L(y)|} 
  \end{equation*}
  \end{minipage}%
  \begin{minipage}{.1\textwidth}
    \centering
    and
  \end{minipage}%
  \begin{minipage}{.4\textwidth}
    \centering
    \begin{equation*}
      \Expect{\totestrips_{v, c}} = \sum_{xyz\in T(\overline{H_c})}\frac{1}{|L(x)||L(y)||L(z)|}.
    \end{equation*}
  \end{minipage}
  
  Combining the above inequalities, we have
  \begin{equation*}
    \Expect{\pairs_{v, X} - \trips_{v,X}} \geq \sum_{c \in L(v)}\sum_{xy\in E(\overline {H_c})}\frac{K}{|L(x)||L(y)|}\left(\frac{99}{100}K - \sum_{\substack{z \in X :\\ xyz\in T(\overline{H_c})}}\frac{1}{3}\cdot\frac{1}{|L(z)|}\right).
  \end{equation*}
  Since $|A| \leq d(v) / 100$ and $|L(v)| \geq (1 - 11\eps)d(v)$, and since $\eps \ll \delta \ll 1$,
  \begin{equation*}
    \sum_{z \in X}\frac{1}{|L(z)|} \leq \frac{|A|}{|L(v)| / 3} + \frac{|B|}{(1 - \delta)|L(v)|} \leq 
    \frac{1}{1 - 11\eps}\left(\frac{3}{100} + \frac{1}{1 - \delta}\right) \leq \frac{104}{100}.
  \end{equation*}
  Combining the previous two inequalities, since $99K / 100 - 104/300 \geq 1 /90$, we have
  \begin{equation}\label{bipartite pairs equation}
    \Expect{\pairs_{v, X} - \trips_{v,X}} \geq \sum_{c \in L(v)}\sum_{xy\in E(\overline {H_c})}\frac{K}{90|L(x)||L(y)|} \geq \frac{K}{|L(v)|^2}\sum_{c \in L(v)}|E(\overline{H_c})| / 100.
  \end{equation}

  Let $H$ be a spanning bipartite subgraph of $\overline {G[A\cup B]}$ with bipartition $(A, B)$ such that each vertex $u\in A$ has at least $\delta d(v) / 20$ neighbors in $B$.  For each $x \in A$ and $y\in B$, let $C_{x, y} = \{c \in L(v) : x, y \in V(H_c)\}$.  We claim that $|C_{x, y}| \geq (1/3 - \delta)|L(v)|$ for each $x \in A$ and $y \in B$.  Suppose $|L(y)| < |L(v)|$, or else $|C_{x, y}| \geq (1 - 2/3)|L(v)| \geq (1/3 - \delta)|L(v)|$ since $(L, M)$ is total, as claimed.  Now $|L(x)| + |L(y)| - |C_{x, y}| \leq |L(v)|$.  Hence, $|C_{x, y}| \geq (1 - 2/3)|L(v)| + (1 - \delta)|L(v)| - |L(v)| = (1/3 - \delta)|L(v)|$, as claimed.
  Now
  \begin{equation*}
    \sum_{c \in L(v)}|E(\overline{H_c})| \geq \sum_{xy \in \overline{H}}|C_{x, y}| \geq (1/3 - \delta)|L(v)|\frac{\delta d(v)}{20}|A|.
  \end{equation*}
  The previous inequality, together with \eqref{bipartite pairs equation}, implies the lemma, since $\mu \ll \delta \ll 1$.
\end{proof}

We will apply the following lemma to $\ordering$-prioritized vertices.

\begin{lemma}\label{subservience savings}
  Let $0 < \mu \ll 1$.  If $\ordering$ is an ordering of $V(G)$ and $v\in V(G)$, then
  \begin{equation*}
    \Expect{\unact_{v, \ordering}} \geq 2\mu|\{u \in N(v) : u \ordering v\}|.
  \end{equation*}
\end{lemma}
\begin{proof}
  Since $\Prob{u\in A} = 99/100$ for each $u\in N(v)$, we have $\Expect{\unact_{v, \ordering}} = |\{u \in N(v) : u \ordering v\}| / 100$ by Linearity of Expectation.  Since $\mu \ll 1$, the result follows.
\end{proof}

\subsection{Applying Theorem~\ref{local sparsity lemma}}

We are finally ready to state the definition of a saved graph.
\begin{define}\label{saved definition}
  We say a graph $G$ with list assignment $L$ is \textit{saved with respect to $L$ and $k$} if
  \begin{itemize}
  \item $k \geq |L(v)|$ for each vertex $v$,
  \item $(1 - 11\varepsilon)d(v) \leq |L(v)| \leq d(v)$ for each vertex $v$, and
  \item there exists an ordering $\ordering$ of $V(G)$ such that every vertex $v$ is either aberrant, slighly aberrant, egalitarian-sparse, bipartite-sparse, or $\ordering$-prioritized with respect to $L$ and $k$.
  \end{itemize}
  We say a subgraph $H\subseteq G$ is \textit{saved with respect to $L$ and $k$} if for every $L$-coloring $\phi$ of $G - V(H)$, the graph $H$ is saved with respect to $L'$ and $k$ where $L'(v) = L(v)\setminus (\cup_{u\in N(v)\setminus V(H)}\phi(u))$.
\end{define}

The following is the main result of this section.  It implies that an $L$-critical graph does not contain a saved subgraph.

\begin{thm}\label{saved is colorable}
  Let $0 < 1/k_0 \ll \eps \ll \mu \ll \delta \ll 1$, and let $k > k_0$.  If $G$ is saved with respect to a list assignment $L$ and $k$, then $G$ is $L$-colorable.
\end{thm}
\begin{proof}
  We assume that $(L, M)$ is a total correspondence assignment for $G$ such that every $(L, M)$-coloring of $G$ is an $L$-coloring.
  We apply Theorem~\ref{local sparsity lemma} with $2k$, $2/3$, $11\varepsilon$, $99/100$, $1$, and $198\eps/100$ playing the roles of $\Delta$, $\sigma$, $\varepsilon$, $\rho$, $\xi_1$, and $\xi_2$, respectively.  Since $d(v) \leq |L(v)|/(1 - 11\varepsilon)$ and $|L(v)| \leq k$ for each vertex $v$, we have $d(v) \leq 2k = \Delta$, as required.  We assume $k$ is large enough so that $(99/100)\log^{10}2k \leq \log^{10}k$ and $2k \geq \Delta_0$ from Theorem~\ref{local sparsity lemma}.  Moreover, since every vertex is either aberrant, slightly aberrant, egalitarian-sparse, bipartite-sparse, or $\ordering$-prioritized with respect to $L$ and $k$, we may assume that the degree of every vertex is at least $1000/(1 - 11\varepsilon)^2$, as required in Theorem~\ref{local sparsity lemma}, and sufficiently large as in Proposition~\ref{prop:keep-squared} and Lemmas~\ref{egal-sparse savings} and~\ref{bipartite-sparse savings}.

  Let $v\in V(G)$.  If $v$ is aberrant or slightly aberrant with respect to $L$ and $k$, then by Lemma~\ref{aberrance savings},
  \begin{equation*}
    \Expect{\savings_{v, \sigma, \ordering}} \geq \Expect{\aberrance_v} \geq 2(\Save_L(v) + \varepsilon\log^{10}k) \geq (1 + \xi_1)\Save_L(v) + \xi_2\log^{10}2k.
  \end{equation*}
  If $v$ is egalitarian-sparse with respect to $L$ and $k$, then by Lemma~\ref{egal-sparse savings}, since $\egal(v) \subseteq \egal_{2/3}(v)$,
  \begin{multline*}
    \Expect{\savings_{v, \sigma, \ordering}} \geq \Expect{\pairs_{v, \egal(v)} - \trips_{v, \egal(v)}} \\\geq 2(\Save_L(v) + \varepsilon\log^{10}k) \geq (1 + \xi_1)\Save_L(v) + \xi_2\log^{10}2k.
  \end{multline*}
  If $v$ is bipartite-sparse with respect to $L$ and $k$, then $v$ has a half-egalitarian bipartition $(A, B)$ such that $|A| \geq 100(\Save_L(v) + \eps \log^{10}k) / \mu$.  Letting $A' \subseteq A$ have size $d(v) / 100 \geq |A| / 100$, we have a half-egalitarian bipartition $(A', B)$ such that $|A'| \geq (\Save_L(v) + \eps \log^{10}k) / \mu$.  Thus, by Lemma~\ref{bipartite-sparse savings}, since $A' \cup B \subseteq \egal_{2/3}(v)$,
  \begin{multline*}
    \Expect{\savings_{v, \sigma, \ordering}} \geq \Expect{\pairs_{v, A' \cup B} - \trips_{v, A' \cup B}} \\
    \geq 2(\Save_L(v) + \varepsilon\log^{10}k) \geq (1 + \xi_1)\Save_L(v) + \xi_2\log^{10}2k.
  \end{multline*}
  If $v$ is $\ordering$-prioritized, then by Lemma~\ref{subservience savings},
    \begin{equation*}
    \Expect{\savings_{v, \sigma, \ordering}} \geq \Expect{\unact_{v, \ordering}} \geq 2(\Save_L(v) + \varepsilon\log^{10}k) \geq (1 + \xi_1)\Save_L(v) + \xi_2\log^{10}2k.
  \end{equation*}
  Therefore $\Expect{\savings_{v, \sigma, \ordering}} \geq \max\{(1 + \xi_1)\Save_L(v), \xi_2\log^{10}2k\}$, as required.

  Now by Theorem~\ref{local sparsity lemma}, $G$ is $(L, M)$-colorable, as desired.
\end{proof}

\section{Finding a saved subgraph}\label{discharging section}

The main result of this section is the following lemma which we prove with discharging.  Recall that a subgraph $H\subseteq G$ is dense with respect to $L$ if there is a matching $M$ in $\overline H$ such that $|E(\overline H)| < |M|(|V(H)| - |M|) - \sum_{u\in V(H)}\Save_L(u)$.

\begin{lemma}\label{inductive discharging lemma}
  Let $G$ be a graph with list assignment $L$ such that \eqref{main thm equation} does not hold and for each vertex $v\in V(G)$, we have $|L(v)| \leq\min\{d(v), k\}$.  If $G$ has no dense subgraph with respect to $L$ and $k$ is sufficiently large, then either
  \begin{enumerate}[(a)]
  \item $G$ is saved with respect to $L$ and $k$, or
  \item\label{discharging outcome} there is a nonempty set $D\subsetneq V(G)$ such that
    \begin{equation*}
      \sum_{v\in V(G-D)}(\Save_L(v) + \varepsilon\log^{10}k) < \sum_{v\in V(G-D)}(2\varepsilon\Gap_{G-D} (v) - 7\varepsilon(k - |L(v)| + |N(v)\cap D|)).
    \end{equation*}
  \end{enumerate}
\end{lemma}

For each $v\in V(G)$, let the \textit{charge} of $v$ be
\begin{equation*}
  \ch(v) = \Save_L(v) - 2\varepsilon\Gap(v) + \varepsilon\log^{10}k + 7\varepsilon(k - |L(v)|).
\end{equation*}
Now, $\sum_{v\in V(G)} \ch(v) \leq 0$.  As mentioned in Section~\ref{overview section}, we think of $D$ in Lemma~\ref{inductive discharging lemma} as the ``discharged set'', i.e.\ the vertices in $D$ will send charge to their neighbors.  When we redistribute the charges in Section~\ref{discharging subsection}, each vertex not in $D$ receives $9\varepsilon$ charge from each neighbor in $D$, and each vertex in $D$ still has positive charge.

In order to prove Lemma~\ref{inductive discharging lemma}, we define the following types of vertices.  Recall that $\eta \in (0, 1)$ satisfies~\eqref{hierarchy}.

\begin{define}\label{heavy definition}\label{very lordly definition}\label{sponsored definition}
  We say a vertex $v \in V(G)$ is
  \begin{itemize}
  \item \textit{heavy} if $\ch(v) \geq \eta \Gap(v)$ and \textit{normal} otherwise,
  \item \textit{extremely heavy} if $\ch(v) > 9\varepsilon d(v)$,
  \item \textit{very lordly} if $\Gap(v) \geq (3\delta/4)d(v)$ and $|\subserv(v)| > \Gap(v)/4$, and
  \item \textit{sponsored} if $v$ has at least $d(v)/2$ heavy neighbors $u$ with $\Save_L(u) \geq 3\eta \Gap(v)$ and $d(u) \leq (1 + \delta)d(v)$.
  \end{itemize}
\end{define}

The next lemma implies that if $v$ is an extremely heavy vertex, then $D = \{v\}$ satisfies~\ref{discharging outcome} in Lemma~\ref{inductive discharging lemma}.  Thus, we can essentially assume there are no extremely heavy vertices.
\begin{lemma}\label{no extremely heavy}
  Let $G$ be a graph with list assignment $L$ such that~\eqref{main thm equation} does not hold.  If $u\in V(G)$ is extremely heavy, then
  \begin{equation*}
    \sum_{v\in V(G-u)}(\Save_L(v) + \varepsilon\log^{10}k) < \sum_{v\in V(G - u)}(2\varepsilon \Gap_{G - u}(v) - 7\varepsilon(k - |L(v)| + |N(v)\cap \{u\}|)).
  \end{equation*}
\end{lemma}
\begin{proof}
  Let $u$ send charge $9\varepsilon$ to each of its neighbors, and denote the resulting charge $\ch_*$. Since $u$ is extremely heavy, $\ch_*(u) > 0$.  Hence,
  \begin{equation*}
    \sum_{v\in V(G - u)}\ch_*(v) < \sum_{v\in V(G)}\ch(v) \leq 0.
  \end{equation*}
  For each vertex $v\in N(u)$, we have $\Gap_{G - u}(v) \geq \Gap_G(v) - 1$.
  Hence,
  \begin{multline*}
    \sum_{v\in V(G - u)}(\Save_L(v) + \varepsilon\log^{10}k) - \sum_{v\in V(G - u)}(2\varepsilon \Gap_{G-u}(v) - 7\varepsilon(k - |L(v)| + |N(v)\cap\{u\}|)) \leq \\
    \sum_{v\in V(G - u)}(\Save_L(v) + \varepsilon\log^{10}k) - \sum_{v\in V(G)\setminus N[u]}(2\varepsilon \Gap_G(v) - 7\varepsilon(k - |L(v)|)) \\
    - \sum_{v\in N(u)}(2\varepsilon(\Gap_{G}(v) - 1) - 7\varepsilon(k - |L(v)| + 1)) = \sum_{v\in V(G - u)}\ch_*(v).
  \end{multline*}
  Now the lemma follows from the previous two inequalities.
\end{proof}

By combining Lemma~\ref{no extremely heavy} with the next lemma, Lemma~\ref{stronger discharging lemma}, we obtain Lemma~\ref{inductive discharging lemma}.  
Recall that a vertex $v\in V(G)$ is \textit{very lordly} if $\Gap(v) \geq (3\delta/4)d(v)$ and $|\subserv(v)| > \Gap(v)/4$.

\begin{lemma}\label{stronger discharging lemma}
  Let $0 < \eps \ll \mu \ll \delta \ll 1$.
  Let $G$ be a graph with list assignment $L$ not satisfying~\eqref{main thm equation} such that for each vertex $v\in V(G)$ we have $|L(v)|\leq\min\{d(v), k\}$.
  Let $S_0$ be the vertices of $G$ that are either aberrant, slightly aberrant, egalitarian-sparse, or bipartite-sparse.  For $i\geq 1$, let $S_i$ be the vertices with at least
\begin{equation*}
  \frac{\Save_L(v) + \varepsilon\log^{10}k}{\mu}
\end{equation*}
neighbors in $\cup_{j=0}^{i - 1}S_{i}$, and define $S_\infty = \cup_{i\geq 1}S_i$.  Let $\mathcal L$ be the very lordly vertices not in $S_\infty$, and let $D = V(G)\setminus (S_\infty \cup \mathcal L)$.
  If $G$ has no extremely heavy vertex and no dense subgraph with respect to $L$, then 
  \begin{equation*}
    \sum_{v\in V(G-D)}(\Save_L(v) + \varepsilon\log^{10}k) < \sum_{v\in V(G-D)}(2\varepsilon\Gap_{G-D} (v) - 7\varepsilon(k - |L(v)| + |N(v)\cap D|)).
  \end{equation*}
\end{lemma}

In Section~\ref{proof of discharging lemma section}, we prove Lemma~\ref{inductive discharging lemma} using Lemmas~\ref{no extremely heavy} and \ref{stronger discharging lemma}.  Sections~\ref{discharging preliminary section}, \ref{discharging structure section}, and \ref{discharging subsection} are devoted to the proof of Lemma~\ref{stronger discharging lemma}.

\subsection{Preliminaries}\label{discharging preliminary section}

Since Sections~\ref{discharging preliminary section}, \ref{discharging structure section}, and \ref{discharging subsection} are devoted to the proof of Lemma~\ref{stronger discharging lemma}, we assume in these sections that $G$ is a graph with list assignment $L$ not satisfying~\eqref{main thm equation} such that $|L(v)| \leq \min\{d(v), k\}$ for each vertex $v$, and moreover $G$ does not contain a dense subgraph or any extremely heavy vertices.  Using this assumption, we prove several useful propositions in this subsection.

We need the following proposition about the sizes of vertices' lists of available colors.  In this proposition, we need that there are no extremely heavy vertices.
\begin{prop}\label{list size proposition}
  Let $0 < \eps \ll 1$.  If $v\in V(G)$, then 
  \begin{enumerate}[(a)]
  \item \label{all lists are linear in d}
    $|L(v)| \geq (1 - 11\varepsilon)d(v)$ and $\Save_L(v) < 11\varepsilon d(v)$, and
  \item \label{all lists are linear in k}
    $|L(v)| > k/3$.
  \end{enumerate}
\end{prop}

\begin{proof}
  First we prove (a).  Since $v$ is not extremely heavy, $9\varepsilon d(v) \geq \ch(v) > \Save_L(v) - 2\varepsilon\Gap(v)$.  Hence, since $\Gap(v) \leq d(v)$, we have $11\varepsilon d(v) > \Save_L(v)$, as desired.  Therefore $d(v) + 1 - |L(v)| < 11\varepsilon d(v)$, so $|L(v)| > (1 - 11\varepsilon)d(v)$, as desired.

  Now we prove (b).  
  Since $v$ is not extremely heavy, $9\varepsilon d(v) \geq \ch(v) > 7\varepsilon(k - |L(v)|) - 2\varepsilon \Gap(v)$.  Hence, since $\Gap(v) \leq d(v)$, we have $11d(v) + 7|L(v)| > 7k$.  By \ref{all lists are linear in d}, $d(v) \leq |L(v)|/(1 - 11\varepsilon)$, and since $\varepsilon \ll 1$, we have $d(v) \leq 14|L(v)|/11$.  Therefore $7k < 11d(v) + 7|L(v)| \leq 21|L(v)|$, so $|L(v)| > k/3$, as desired.
\end{proof}

Proposition~\ref{list size proposition} \ref{all lists are linear in k} reveals why we need to add the term $7\varepsilon(k - |L(v)|)$ in Theorem~\ref{local critical thm}.  Note that Proposition~\ref{list size proposition} \ref{all lists are linear in k} implies that all neighbors of a vertex are $2/3$-egalitarian.  This fact will be crucial in Lemma~\ref{normal vertex structure lemma} \ref{normal small gap and large subservient implies bipartite-sparse}.

The next proposition provides useful facts about the heavy vertices.
\begin{prop}\label{heavy vertex facts}
  Let $0 < \eps \ll \eta \ll \delta \ll 1$.
  If $v\in V(G)$ is heavy, then
  \begin{enumerate}[(a)]
  \item $\Gap(v) \leq (\delta/4)d(v)$\label{heavy has small gap}, and 
  \item \label{heavy charge linear in save}
    $\ch(v) > \frac{\Save_L(v) + \varepsilon\log^{10}k}{1 + \eta}.$
  \end{enumerate}
\end{prop}
\begin{proof}
  First we prove (a).  Since $v$ is not extremely heavy, $\ch(v) \leq 9\varepsilon d(v)$.  Since $\ch(v) \geq \eta\Gap(v)$, we have $\Gap(v) \leq 9\eps d(v) / \eta \leq (\delta/4)d(v)$, as desired.

  Now we prove (b).
  Since $v$ is heavy, $\ch(v) > \eta\Gap(v)$.  Hence, $2\varepsilon\Gap(v) < 2\eps \ch(v) / \eta < \eta \ch(v)$.  Therefore
  \begin{equation*}
    \ch(v) > \Save_L(v) - \eta \ch(v) + \varepsilon\log^{10}k + 7\varepsilon(k - |L(v)|),
  \end{equation*}
  and the result follows by rearranging terms.
\end{proof}

The heavy vertices in $D$ will send charge to their neighbors in $S_\infty$.  Assuming $\varepsilon$ is small enough, Proposition~\ref{heavy vertex facts} \ref{heavy charge linear in save} implies that these vertices will have plenty of charge to send to these neighbors.  Proposition~\ref{heavy vertex facts} \ref{heavy has small gap}, in conjunction with Lemma~\ref{in big clique with many very lordly nbrs is aberrant}, implies that heavy vertices with many very lordly neighbors are aberrant.  Thus, heavy vertices in $D$ do not have to send too much charge to very lordly neighbors.

The next proposition implies that if $v$ is a normal vertex, then $\Save_L(v) + \varepsilon\log^{10}k$ is a fraction of $\Gap(v)$.  Thus, the main result in~\cite[Theorem 1.7]{KP18} implies that if every vertex is normal, then for $\varepsilon$ small enough, the graph is $L$-colorable.
\begin{prop}\label{normal Gap linear in save}
  Let $0 < \eps \ll \eta \ll 1$.
  If $v\in V(G)$ is normal, then
  \begin{equation*}
    \Gap(v) \geq \frac{\Save_L(v) + \varepsilon\log^{10}k}{\eta + 2\eps}.
  \end{equation*}
\end{prop}
\begin{proof}
  Since $v$ is normal, $\ch(v) \leq \eta\Gap(v)$, and since $|L(v)| \leq k$, we have $\ch(v) \geq \Save_L(v) - 2\varepsilon\Gap(v) + \varepsilon \log^{10} k$.  Therefore $\Save_L(v) + \varepsilon\log^{10} k \leq \eta \Gap(v) + 2\varepsilon\Gap(v)$, and the result follows by rearranging terms.
\end{proof}

\subsection{Structure}\label{discharging structure section}

In this subsection we prove Lemma~\ref{normal vertex structure lemma}, which implies that every normal vertex not in $S_\infty$ is either very lordly or sponsored (recall that a vertex $v\in V(G)$ is \textit{sponsored} if it has at least $d(v)/2$ heavy neighbors $u$ with $\Save_L(u) \geq 3\eta\Gap(v)$ and $d(u) \leq (1 + \delta)d(v)$).  In the latter case the vertex is in $D$, and the charge it receives from its heavy neighbors compensates for the charge it sends to its neighbors not in $D$.  We also prove Lemma~\ref{in big clique with many very lordly nbrs is aberrant}, which implies that a vertex in $D$ does not have too many very lordly neighbors.  In Section~\ref{discharging subsection}, we use these two lemmas to show that after redistributing charges, the vertices in $D$ all have positive charge.

  

\begin{lemma}\label{normal vertex structure lemma}
  Let $0 < \eps \ll \eta \ll \mu \ll \delta \ll 1$.  Let $v\in V(G)$ be a normal vertex.
  \begin{enumerate}[(a)]
  \item If $\Gap(v) \geq (3\delta/4)d(v)$, then $v$ is either aberrant, egalitarian-sparse, or very lordly.\label{normal with large gap is saved or lordly}
  \item If $\Gap(v) \in [(\delta/4)d(v), (3\delta/4)d(v))$, then $v$ is either aberrant, bipartite-sparse, or egalitarian-sparse.\label{normal with medium gap is saved}
  \item If $\Gap(v) \leq (\delta/4)d(v)$, then $v$ is either aberrant, slightly aberrant, bipartite-sparse, egalitarian-sparse, or sponsored.\label{normal with small gap is saved or loaded}
  \end{enumerate}
\end{lemma}

The proof of Lemma~\ref{normal vertex structure lemma} comprises most of this subsection.  First, we state the other important lemma of this subsection, which we prove at the end of the subsection, after first proving Proposition~\ref{very lordly nbr in big clique is lordlier}, which implies that many very lordly neighbors of a vertex $v$ that is heavy or sponsored are lordlier.

\begin{lemma}\label{in big clique with many very lordly nbrs is aberrant}
  Let $0 < \eps \ll \mu  \ll 1$.  Let $v\in V(G)$ be a vertex such that $\Gap(v) \leq (\delta/4)d(v)$.  If $v$ has at least
  \begin{equation*}
    \frac{\Save_L(v) + \varepsilon\log^{10}k}{\mu} + \Gap(v)
  \end{equation*}
  very lordly neighbors, then $v$ is aberrant.
\end{lemma}

The following lemma will be used to prove Lemma~\ref{normal vertex structure lemma}.
\begin{lemma}\label{normal refined structure}
  Let $0 < \eps \ll \eta \ll \mu \ll \delta \ll 1$.  Let $v\in V(G)$ be a normal vertex.
  \begin{enumerate}[(a)]
  \item \label{normal and subservient implies aberrance}
    If either $|\aberrant(v)| \geq \Gap(v)/4$ or $|\slaberrant(v)| \geq d(v)/4$, then $v$ is aberrant or slightly aberrant.
  \item \label{normal small gap and large subservient implies bipartite-sparse}
    If $\Gap(v) < (3\delta / 4)d(v)$, $|\aberrant(v)| < \Gap(v)/4$, and $|\subserv(v)| \geq \Gap(v) / 4$, then $v$ is bipartite-sparse.
  \item \label{normal large egal implies aberrance}
    If $\Gap(v) \geq (\delta/4)d(v)$ and $|\subserv(v)\cup\aberrant(v)| < \Gap(v)/2$,  then $v$ is egalitarian-sparse.
  \item \label{normal small gap small subserv small lordly}
    If $\Gap(v) \leq (\delta/4)d(v)$, $|\subserv(v)\cup\aberrant(v)| < \Gap(v)/2$, and $|\slaberrant(v)| < d(v) / 4$, then $v$ is either egalitarian-sparse or sponsored.
  \end{enumerate}
\end{lemma}
\begin{proof}[Proof of Lemma~\ref{normal refined structure} \ref{normal and subservient implies aberrance}]
  First, suppose $|\aberrant(v)| \geq \Gap(v)/4$.  Hence, since $v$ is normal, by Proposition~\ref{normal Gap linear in save}, since $\varepsilon \ll \eta \ll \mu$,
  \begin{equation*}
    |\aberrant(v)| \geq \frac{\Save_L(v) + \varepsilon\log^{10}k}{4\eta + 8\eps} \geq \frac{\Save_L(v) + \varepsilon\log^{10}k}{\mu},
  \end{equation*}
  so $v$ is aberrant, as desired.

    Therefore we may assume $|\slaberrant(v)| \geq d(v)/4 = (\Gap(v)d(v))/(4\Gap(v)).$  By Proposition~\ref{normal Gap linear in save}, since $\varepsilon \ll \eta \ll \mu$,
    \begin{equation*}
      |\slaberrant(v)| \geq \left(\frac{d(v)}{\Gap(v)}\right)\left(\frac{\Save_L(v) + \varepsilon\log^{10}k}{4\eta + 8\eps}\right) \geq \left(\frac{d(v)}{\Gap(v)}\right)\left(\frac{\Save_L(v) + \varepsilon\log^{10}k}{\mu}\right),
    \end{equation*}
    so $v$ is slightly aberrant, as desired.
\end{proof}

\begin{proof}[Proof of Lemma~\ref{normal refined structure} \ref{normal small gap and large subservient implies bipartite-sparse}]
  Let $B$ be a maximum cardinality clique in $G[N(v)\setminus \aberrant(v)]$.  Note that $|B| \geq \omega(v) - 1 - |\aberrant(v)|$.  Since $\Gap(v) < (3\delta/4)d(v)$, we have $\omega(v) - 1 \geq (1 - 3\delta/4)d(v)$, and since $|\aberrant(v)| < \Gap(v)/4 < (3\delta/16)d(v)$, we have
  \begin{equation}\label{B is large}
    |B| \geq (1 - 15\delta/16)d(v).
  \end{equation}
  Let $A = \subserv(v)$.  We claim that $(A, B)$ is a half-egalitarian bipartition for $v$.  By Proposition~\ref{list size proposition} \ref{all lists are linear in k}, $A\subseteq \egal_{2/3}(v)$.  Since $\varepsilon \ll \delta \ll 1$, by Proposition~\ref{list size proposition} \ref{all lists are linear in d} and~\eqref{B is large}, for each $u\in B$, we have $|L(u)| \geq (1 - 11\varepsilon)|B| \geq (1 - \delta)d(v) \geq (1 - \delta)|L(v)|$.  Hence, $B\subseteq \egal(v)$.  By Proposition~\ref{list size proposition} \ref{all lists are linear in d}, for each $u\in A$, we have $d(u) \leq (1 - \delta)|L(v)|/(1 - 11\varepsilon) \leq (1 - \delta)d(v)/(1 - 11\varepsilon)$.  Hence, by~\eqref{B is large}, each $u\in A$ has at least $(1 - 15\delta/16 - (1 - \delta)/(1 - 11\varepsilon))d(v) \geq \delta d(v) / 20$ non-neighbors in $B$, as required.  Therefore $(A, B)$ is a half-egalitarian bipartition for $v$, as claimed.

  Since $|A| = |\subserv(v)| \geq \Gap(v) / 4$, by Proposition~\ref{normal Gap linear in save},
  \begin{equation*}
    |A| \geq \frac{\Save_L(v) + \varepsilon\log^{10}k}{4\eta + 8\eps} \geq 100\cdot\frac{\Save_L(v) + \varepsilon\log^{10}k}{\mu},
  \end{equation*}
  so $v$ is bipartite-sparse, as desired.
\end{proof}

Lemma~\ref{normal refined structure} \ref{normal and subservient implies aberrance} and \ref{normal small gap and large subservient implies bipartite-sparse} together imply that if a vertex $v\in D$ satisfies $\Gap(v) < (3\delta/4)d(v)$, then $v$ has many egalitarian neighbors.  Our next goal is to prove that if these vertices are not egalitarian-sparse, they have many heavy neighbors.  We use the fact that the egalitarian neighbors of a vertex do not induce a dense subgraph with respect to $L$, so it will be useful to bound the value of $\Save_L$, as in the next two propositions.

\begin{prop}\label{egal nbr save bound}
  Let $0 < \eps \ll \delta \ll 1$.
  If $u$ is an egalitarian neighbor of a vertex $v$ (i.e.\ $u\in \egal(v)$), then
  \begin{equation*}
    \Save_L(u) \leq 12 \varepsilon d(v).
  \end{equation*}
\end{prop}
\begin{proof}
  Since $u\in \egal(v)$, by Proposition~\ref{list size proposition} \ref{all lists are linear in d},
  \begin{equation*}
    d(u) \leq \frac{|L(u)|}{1 - 11\eps} \leq \frac{1 + \delta / 3}{1 - 11 \eps}|L(v)| \leq \frac{1 + \delta / 3}{1 - 11\eps}d(v) \leq \frac{12}{11}d(v),
  \end{equation*}
  so again by Proposition~\ref{list size proposition} \ref{all lists are linear in d}, $\Save_L(u) \leq 11\eps d(u) \leq 12\eps d(v)$, as desired.
\end{proof}

\begin{prop}\label{large antimatching}
  Let $0 < \eps \ll \eta \ll \delta \ll 1$.  Let $v\in V(G)$ such that $|\subserv(v)\cup \aberrant(v)| < \Gap(v)/2$.  If $M$ is a maximum matching in $\overline{G[\egal(v)]}$, then $|M|\geq\Gap(v)/4$.  Furthermore, if $u\in \egal(v)\setminus (V(M)\cup \slaberrant(v))$ is normal, then
  \begin{equation*}
    \Save_L(u) \leq 3\eta\Gap(v).
  \end{equation*}
\end{prop}
\begin{proof}
  Let $M$ be a maximum matching in $\overline{G[\egal(v)]}$.  By the choice of $M$, the vertices of $\egal(v)\setminus V(M)$ form a clique.  Therefore
  \begin{equation*}
    2|M| + |\subserv(v)\cup\aberrant(v)| + \omega(v) - 1 \geq d(v).
  \end{equation*}
  Similarly, since $u\in \egal(v)\setminus V(M)$,
  \begin{equation}
    \label{u clique lower bound}
    \omega(u) \geq d(v) + 1 - 2|M| - |\subserv(v)\cup\aberrant(v)|.
  \end{equation}
  Hence, since $|\subserv(v)\cup\aberrant(v)| < \Gap(v)/2$,
  \begin{equation*}
    |M| \geq \Gap(v)/4,
  \end{equation*}
  as desired.

  Since no clique in $G[\egal(v)]$ contains an edge in $M$, $\omega(G[\egal(v)]) \leq |\egal(v)| - |M|$.  Note that for any $H\subseteq G[N(v)\cup\{v\}]$, $|V(H)| - \omega(H) \leq \Gap(v)$.  Hence,
  \begin{equation}\label{antimatching upper bound}
    |M| \leq \Gap(v).
  \end{equation}
  By \eqref{u clique lower bound} and \eqref{antimatching upper bound}, since $|\subserv(v)\cup\aberrant(v)| < \Gap(v)/2$,
  \begin{equation}
    \label{u simplified clique lower bound}
    \omega(u) \geq d(v) + 1 - (5/2)\Gap(v).
  \end{equation}
  Since $|L(u)| = d(u) + 1 -\Save_L(u) = \Gap(u) + \omega(u) - \Save_L(u)$, by \eqref{u simplified clique lower bound},
  \begin{equation*}
    |L(u)| \geq \Gap(u) - \Save_L(u) + d(v) + 1 - (5/2)\Gap(v).
  \end{equation*}
  Since $u\notin \slaberrant(v)$, we have $|L(u)| \leq |L(v)| + \delta\Gap(v)/3$.  Hence,
  \begin{equation*}
    |L(v)| + (5/2 + \delta / 3)\Gap(v) \geq \Gap(u) - \Save_L(u) + d(v) + 1.
  \end{equation*}
  Since $d(v) + 1 - |L(v)| = \Save_L(v)$, we have
  \begin{equation}\label{strong egal nbr gap bound}
    \Gap(u) - \Save_L(u) \leq (5/2 + \delta / 3)\Gap(v) - \Save_L(v) \leq 8\Gap(v) / 3.
  \end{equation}
  Since $u$ is normal, by Proposition~\ref{normal Gap linear in save},
  \begin{equation}
    \label{normal strong egal nbr save bound}
    \Gap(u) \geq \frac{\Save_L(v)}{\eta + 2\eps}.
  \end{equation}
  Combining~\eqref{strong egal nbr gap bound} and \eqref{normal strong egal nbr save bound}, we obtain
  \begin{equation*}
    \Save_L(u)\left(\frac{1}{\eta + 2\eps} - 1\right) \leq 8\Gap(v) / 3.
  \end{equation*}
  By rearranging terms in the previous inequality, since $\eps \ll \mu \ll 1$, we obtain the desired bound on $\Save_L(u)$.
\end{proof}

Now we can prove Lemma~\ref{normal refined structure} \ref{normal large egal implies aberrance} and \ref{normal small gap small subserv small lordly}.

\begin{proof}[Proof of Lemma~\ref{normal refined structure} \ref{normal large egal implies aberrance}]
    Let $H = G[\egal(v)]$.  By Proposition~\ref{large antimatching}, there is a maximum matching $M$ in $\overline{H}$ such that $|M|\geq \Gap(v)/4$.
  Since $G$ contains no dense subgraph with respect to $L$,
  \begin{equation}\label{large gap egal-sparse}
    |E(\overline H)| \geq |M|(|V(H)| - |M|) - \sum_{u\in V(H)}\Save_L(u).
  \end{equation}
  Therefore by \eqref{large gap egal-sparse} and Proposition~\ref{egal nbr save bound},
  \begin{equation*}
    |E(\overline H)| \geq |M|(|V(H)| - |M|) - 12\eps |V(H)| d(v) = |V(H)|(|M| - 12\eps d(v)) - |M|^2.
  \end{equation*}

  Note that the function $x \mapsto ax - x^2$ is increasing for $x < a / 2$, and $|M| \leq |V(H)| / 2$.  Thus,
  since $|V(H)| \geq d(v) - \Gap(v)/2$ and $|M| \geq \Gap(v)/4$, the inequality above implies that
  \begin{align*}
    |E(\overline H)| &\geq (d(v) - \Gap(v)/2)(\Gap(v)/4 - 12 \varepsilon d(v)) - (\Gap(v)/4)^2\\
                     &\geq\Gap(v)d(v)\left(1/4 - 6\eps\right) - (3/16)\Gap(v)^2 - 12\eps d(v)^2.
  \end{align*}
  Since $\Gap(v) \leq d(v)$, we have $(3/16)\Gap(v)^2 \leq (3/16)\Gap(v)d(v)$.  Therefore since $\Gap(v) \geq (\delta/4)d(v)$,
  \begin{align*}
    |E(\overline H)| &\geq \Gap(v)d(v)\left(1/16 - 6\eps\right) - 12\eps d(v)^2 \geq d(v)^2\left((\delta/4)(1/16 - 6\eps) - 12\eps\right) \geq \delta d(v)^2 / 65.
  \end{align*}
  Hence, by Proposition~\ref{normal Gap linear in save}, since $\eps \ll \eta\ll \mu \ll \delta$,
  \begin{equation*}
    |E(\overline H)| \geq d(v)\left(\frac{\Save_L(v) + \varepsilon\log^{10}k}{\eta + 2\eps}\right)\delta/65 \geq d(v)\left(\frac{\Save_L(v) + \varepsilon\log^{10}k}{\mu}\right),
  \end{equation*}
  so $v$ is egalitarian-sparse, as desired.
\end{proof}

\begin{proof}[Proof of Lemma~\ref{normal refined structure} \ref{normal small gap small subserv small lordly}]
    By Proposition~\ref{large antimatching}, there is a maximum matching $M$ in $\overline{G[\egal(v)]}$ such that $|M| \geq \Gap(v) / 4$.  Let $H$ be the graph induced by $\overline G$ on vertices in $V(M)$ and vertices $u \in \egal(v)\setminus (V(M)\cup \slaberrant(v))$ such that $\Save_L(u) \leq 3\eta\Gap(v)$.  By Proposition~\ref{large antimatching}, if $u\in \egal(v)\setminus V(H)$, then either $u\in \slaberrant(v)$ or $u$ is heavy.  By the choice of $H$, if $u\in \egal(v)\setminus V(H)$ is heavy, then $\Save_L(u) \geq 3\eta\Gap(v)$.  Hence, since $|\slaberrant(v)| < d(v)/4$ and $|\subserv(v)\cup\aberrant(v)| < \Gap(v)/2 \leq \delta d(v) / 8 \leq d(v) / 4$, we have
  \begin{equation}
    \label{small gap H lower bound}
    |V(H)| \geq d(v)/4,
  \end{equation}
  or else $v$ has at least $d(v)/2$ heavy egalitarian neighbors with $\Save_L(u) \geq 3\eta\Gap(v)$ and $d(u) \leq |L(u)|/(1 - 11\varepsilon) \leq \frac{|L(v)| + \delta\Gap(v)/3}{1 - 11\varepsilon} \leq \frac{1 + \delta^2/12}{1 - 11\varepsilon}d(v) \leq (1 + \delta)d(v)$ and is thus sponsored, as desired.

  Since $G$ contains no dense subgraph with respect to $L$,
  \begin{equation}\label{small gap egal-sparse}
    |E(\overline H)| \geq |M|(|V(H)| - |M|) - \sum_{u\in V(H)}\Save_L(u).
  \end{equation}
  By the choice of $H$,
  \begin{equation}
    \label{not matching save bound}
    \sum_{u\in V(H)\setminus V(M)}\Save_L(u) \leq (|V(H)| - 2|M|)3\eta\Gap(v).
  \end{equation}
  By Proposition~\ref{egal nbr save bound},
  \begin{equation}\label{matching save bound}
    \sum_{u\in V(M)}\Save(u) \leq 24\varepsilon|M|d(v).
  \end{equation}
  By~\eqref{small gap egal-sparse}, \eqref{not matching save bound}, and \eqref{matching save bound},
  \begin{equation}
    \label{small gap combined sparsity bound}
    |E(\overline H)| \geq (|V(H)| - |M|)\left(|M| - 3\eta\Gap(v)\right) - 24\varepsilon|M|d(v).
  \end{equation}
  By~\eqref{small gap H lower bound} and \eqref{small gap combined sparsity bound}, since $\Gap(v)/4 \leq |M| \leq \Gap(v) \leq (\delta/4)d(v)$,
  \begin{equation*}
    |E(\overline H)| \geq \left(\frac{1 - \delta}{4}\right)\left(\frac{1}{4} - 3\eta - 24\varepsilon\right)\Gap(v)d(v) \geq \left(\frac{1 - \delta}{17}\right)\Gap(v)d(v).
  \end{equation*}
  Hence, by Proposition~\ref{normal Gap linear in save}, since $\eps \ll \eta \ll \mu \ll \delta$,
  \begin{equation*}
    |E(\overline H)| \geq \left(\frac{1 - \delta}{17}\right)d(v)\frac{\Save_L(v) + \varepsilon \log^{10}k}{\eta + 2\eps} \geq \frac{\Save_L(v) + \varepsilon \log^{10}k}{\mu},
  \end{equation*}
  so $v$ is egalitarian-sparse, as desired.
\end{proof}

Now we use Lemma~\ref{normal refined structure} to prove Lemma~\ref{normal vertex structure lemma}.

\begin{proof}[Proof of Lemma~\ref{normal vertex structure lemma}]
  First we prove (a).  Assume $v$ is not very lordly.  Hence, $|\subserv(v)| \leq \Gap(v) / 4$.  By Lemma~\ref{normal refined structure} \ref{normal and subservient implies aberrance}, we may assume $|\aberrant(v)| < \Gap(v) / 4$, or else $v$ is aberrant, as desired.  Hence, $|\subserv(v)\cup\aberrant(v)| < \Gap(v)/2$.  Therefore, by Lemma~\ref{normal refined structure} \ref{normal large egal implies aberrance}, $v$ is egalitarian-sparse, as desired.
  
  Next we prove (b).
  By Lemma~\ref{normal refined structure} \ref{normal and subservient implies aberrance}, we may assume $|\aberrant(v)| < \Gap(v) / 4$, or else $v$ is aberrant, as desired.  By Lemma~\ref{normal refined structure} \ref{normal small gap and large subservient implies bipartite-sparse}, we may assume $|\subserv(v)| \leq \Gap(v)/4$, or else $v$ is bipartite-sparse, as desired.  Therefore $|\subserv(v)\cup\aberrant(v)| < \Gap(v)/2$.  Hence, by Lemma~\ref{normal refined structure} \ref{normal large egal implies aberrance}, $v$ is egalitarian-sparse, as desired.

  Now we prove (c).
  By Lemma~\ref{normal refined structure} \ref{normal and subservient implies aberrance}, we may assume that $|\aberrant(v)| < \Gap(v) / 4$ and also $|\slaberrant(v)| < d(v) / 4$, or else $v$ is aberrant or slightly aberrant, as desired.  By Lemma~\ref{normal refined structure} \ref{normal small gap and large subservient implies bipartite-sparse}, we may assume $|\subserv(v)| \leq \Gap(v)/4$, or else $v$ is bipartite-sparse, as desired.  Therefore $|\subserv(v)\cup\aberrant(v)| < \Gap(v)/2$.  By Lemma~\ref{normal refined structure} \ref{normal small gap small subserv small lordly}, $v$ is either egalitarian-sparse or sponsored, as desired.
\end{proof}

In the remainder of this subsection, we prove Lemma~\ref{in big clique with many very lordly nbrs is aberrant}.  First we need the following proposition, which implies that many very lordly neighbors of a vertex $v$ that is heavy or sponsored are lordlier. 
\begin{prop}\label{very lordly nbr in big clique is lordlier}
  Let $0 < \eps \ll \delta \ll 1$.  Let $v\in V(G)$, and let $u\in N(v)$ be a very lordly vertex.  If $\omega(u) \geq (1 - \delta/4)d(v) + 1$, then $u\in \aberrant(v)$.
\end{prop}
\begin{proof}
  Since $d(u) + 1 = \Gap(u) + \omega(u)$, $\Gap(u) \geq (3\delta/4)d(u)$, and $\omega(u) \geq (1 - \delta/4)d(v) + 1$, we have $d(u) + 1 \geq (3\delta/4)d(u) + (1 - \delta/4)d(v) + 1$.  Hence,
  \begin{equation*}
    (1 - 3\delta/4)d(u) \geq (1 - \delta/4)d(v) \geq (1 - \delta/4)|L(v)|.
  \end{equation*}
  By Proposition~\ref{list size proposition} \ref{all lists are linear in d} and the previous inequality,
  \begin{equation*}
    |L(u)| \geq \frac{(1 - 11\varepsilon)(1 - \delta/4)}{1 - 3\delta/4}|L(v)|.
  \end{equation*}
  Note that since $\eps \ll \delta \ll 1$,
  \begin{equation*}
    \frac{(1 - 11\varepsilon)(1 - \delta/4)}{1 - 3\delta/4} = 1 + \frac{\delta(2 + 11\varepsilon) - 44\varepsilon}{4 - 3\delta} \geq 1 + \delta / 3.
  \end{equation*}
  By the previous two inequalities, $|L(u)| \geq (1 + \delta / 3)|L(v)|$, so $u\in \aberrant(v)$, as desired.
\end{proof}

\begin{proof}[Proof of Lemma~\ref{in big clique with many very lordly nbrs is aberrant}]
  Since $\Gap(v) \leq (\delta/4)d(v)$, we have $\omega(v) \geq (1 - \delta/4)d(v) + 1$.  By assumption, $v$ has at least $(\Save_L(v) + \varepsilon\log^{10}k)) / \mu$ very lordly neighbors $u$ such that $\omega(u) \geq \omega(v)$.  Hence, by Proposition~\ref{very lordly nbr in big clique is lordlier}, each such very lordly neighbor $u$ is in $\aberrant(v)$.  Therefore $v$ is aberrant, as desired.
\end{proof}

\subsection{Discharging}\label{discharging subsection}

In this subsection we prove Lemma~\ref{stronger discharging lemma} using discharging.  We redistribute the charges sequentially according to the following rules.  Let $v\in D$.
\begin{enumerate}[(R1)]
\item If $v$ is heavy, then $v$ sends $9\varepsilon$ charge to each neighbor not in $D$.  Denote the new charges $\ch_1$.
\item If $v$ is heavy, then $v$ sends $\ch(v)/(2(|N(v)\cap D|))$ to each neighbor in $D$.  Denote the new charges $\ch_2$.
\item[] If $v$ is normal, then
\item $v$ sends $9\varepsilon$ to each neighbor in $S_\infty$, and
\item $v$ sends $9\varepsilon$ to each neighbor in $\mathcal L$.
\end{enumerate}
Denote the final charges $\ch_*$.  
\begin{prop}
  If $v\in S_\infty \cup \mathcal L$, then the final charge is
  \begin{equation*}
    \ch_*(v) = \ch(v) + 9\varepsilon(|N(v)\cap D|).
  \end{equation*}
\end{prop}
\begin{proof}
  If $v\in S_\infty$, then $v$ receives $9\varepsilon$ charge from each neighbor in $D$ under R1 and R3.  If $v\in \mathcal L$, then $v$ receives $9\varepsilon$ charge from each neighbor in $D$ under R1 and R4.
\end{proof}

Our aim is to show that for each vertex $v$ in $D$, we have $\ch_*(v) > 0$.  The next lemma implies this result for heavy vertices in $D$.

\begin{lemma}\label{heavy sends less than half charge outside D}
  Let $0 < \eps \ll \eta \ll \mu \ll \delta \ll 1$.  If $v\in D$ is heavy, then $\ch_1(v) > \ch(v)/2$. 
\end{lemma}
\begin{proof}
  It suffices to show that $v$ sends less than $\ch(v)/2$ charge under R1.  Since $v\notin S_\infty$, at most $(\Save_L(v) + \varepsilon\log^{10}k) / \mu$ neighbors of $v$ are in $S_\infty$, so by Proposition~\ref{heavy vertex facts} \ref{heavy has small gap} and Lemma~\ref{in big clique with many very lordly nbrs is aberrant}, $v$ has at most $(\Save_L(v) + \varepsilon\log^{10}k)/\mu + \Gap(v)$ neighbors in $\mathcal L$.  Therefore $v$ sends at most $9\varepsilon\left(2\left(\Save_L(v) + \varepsilon\log^{10}k\right)/\mu + \Gap(v)\right)$ charge under R1.

  Since $\ch(v) \geq \eta \Gap(v)$, we have $9\varepsilon\Gap(v) \leq 9\eps\ch(v) / \eta < \ch(v)/4$.  By Proposition~\ref{heavy vertex facts} \ref{heavy charge linear in save}, $18\varepsilon\left(\Save_L(v) + \varepsilon\log^{10}k\right) / \mu \leq 18\eps(1 + \eta)\ch(v) / \mu < \ch(v) / 4.$
  Therefore $v$ sends at most $\ch(v)/2$ charge under R1, as desired.
\end{proof}

Now we show that normal vertices in $D$ also have positive final charge.
\begin{lemma}\label{normal in D has positive final charge}
  Let $0 < \eps \ll \eta \ll \mu \ll \delta \ll 1$.  If $v\in D$ is normal, then $\ch_*(v) > 0$.
\end{lemma}
\begin{proof}
  By Lemma~\ref{normal vertex structure lemma}, $\Gap(v) < (\delta/4)d(v)$ and $v$ is sponsored, that is there is a set $X$ of at least $d(v)/2$ heavy neighbors $u$ with $\Save_L(u) \geq 3\eta\Gap(v)$.

  Since $v\notin S_\infty$, at most $(\Save_L(v) + \varepsilon\log^{10}k) / \mu$ neighbors of $v$ are in $S_\infty$.  In particular,
  \begin{equation}\label{not many saved heavy nbrs}
    |X\cap S_\infty| \leq \frac{\Save_L(v) + \varepsilon\log^{10}k}{\mu}.
  \end{equation}
  By Proposition~\ref{normal Gap linear in save}, since $\Gap(v) < (\delta/4)d(v)$,
  \begin{equation}\label{save small compared to degree}
    \Save_L(v) + \varepsilon\log^{10}k < (\eta + 2\eps)(\delta/4)d(v) < \eta d(v).
  \end{equation}
  Combining~\eqref{not many saved heavy nbrs} and \eqref{save small compared to degree}, since $|X| \geq d(v)/2$,
  \begin{equation}\label{most heavy nbrs in D}
    |X\setminus S_\infty| > \left(\frac{1}{2} - \frac{\eta}{\mu}\right)d(v) \geq d(v) / 3.
  \end{equation}

  By Proposition~\ref{heavy vertex facts} \ref{heavy charge linear in save}, each vertex $u\in X\setminus S_\infty$ sends at least
  \begin{equation}\label{charge from heavy nbrs}
    \frac{\ch(u)}{2d(u)} \geq \frac{\Save_L(u)}{2(1 + \eta)d(u)} \geq \frac{3\eta\Gap(v)}{2(1 + \eta)(1 + \delta)d(v)} \geq \frac{\eta\Gap(v)}{d(v)}
  \end{equation}
  charge under R2.

  By~\eqref{most heavy nbrs in D} and \eqref{charge from heavy nbrs}, $v$ receives greater than $\eta\Gap(v) / 3$
  charge under R2.  Since $ch(v) \geq -2\varepsilon\Gap(v)$,
  \begin{equation}
    \label{normal charge received}
    \ch_2(v) > (\eta / 3 - 2\eps)\Gap(v) > \eta\Gap(v) / 4.
  \end{equation}

  By Lemma~\ref{in big clique with many very lordly nbrs is aberrant}, $v$ has at most $(\Save_L(v) + \varepsilon\log^{10}k) / \mu + \Gap(v)$ neighbors in $L$.  Since $v$ has at most $(\Save_L(v) + \varepsilon\log^{10}k) / \mu$ neighbors of $v$ are in $S_\infty$, $v$ sends at most
  \begin{equation*}
    9\varepsilon\left(2\left(\Save_L(v) + \varepsilon\log^{10}k\right) / \mu + \Gap(v)\right)
  \end{equation*}
  charge under R3 and R4.  Hence, by Proposition~\ref{normal Gap linear in save}, 
  \begin{equation}
    \label{normal charge sent}
    \ch_2(v) - \ch_*(v) \leq 9\varepsilon\left(1 + \eta + 2\eps\right)\Gap(v) \leq 10\eps \Gap(v).
  \end{equation}
  By combining~\eqref{normal charge received} and \eqref{normal charge sent}, $\ch_*(v) > \left(\eta /4 - 10\eps\right)\Gap(v) \geq 0$,
  as desired.
\end{proof}

We can finally prove Lemma~\ref{stronger discharging lemma}.
\begin{proof}[Proof of Lemma~\ref{stronger discharging lemma}]

  Recall that
  \begin{equation}\label{final charge still at most 0}
    \sum_{v\in V(G)}\ch_*(v) = \sum_{v\in V(G)}\ch(v) \leq 0.
  \end{equation}
  By Lemmas~\ref{heavy sends less than half charge outside D} and \ref{normal in D has positive final charge}, if $v\in D$, then $\ch_*(v) > 0$.  Therefore $\sum_{v\in V(G)\setminus D}\ch_*(v) > 0$, and thus $D\subsetneq V(G)$.  Note that for each $v\in V(G)\setminus D$, we have
  \begin{multline}\label{charge pays}
    \ch_*(v) = \ch(v) + 9\varepsilon|N(v)\cap D| \\
    \geq \Save_L(v) + \varepsilon\log^{10}k - 2\varepsilon\Gap_{G - D}(u) + 7\varepsilon(k - |L(v)| +|N(v)\cap D|).
  \end{multline}
  Combining~\eqref{final charge still at most 0} and \eqref{charge pays}, we have
  \begin{equation*}
    \sum_{v\in V(G)\setminus D}(\Save_L(v) + \varepsilon\log^{10} k) \leq \sum_{v\in V(G)\setminus D}(2\varepsilon\Gap_{G - D} - 7\varepsilon(k - |L(v)| + |N(v)\cap D|),
  \end{equation*}
  as desired.
\end{proof} 

\subsection{Proof of Lemma~\ref{inductive discharging lemma}}\label{proof of discharging lemma section}

We conclude this section with the proof of Lemma~\ref{inductive discharging lemma} using Lemmas~\ref{no extremely heavy} and \ref{stronger discharging lemma}.

\begin{proof}[Proof of Lemma~\ref{inductive discharging lemma}]
  Let $D, \mathcal L$, and $S_\infty$ be as defined in Lemma~\ref{stronger discharging lemma}.  By Lemma~\ref{no extremely heavy}, we may assume $G$ has no extremely heavy vertices, or else (b) holds, as desired.  Therefore, by Lemma~\ref{stronger discharging lemma}, $D = \varnothing$, or else (b) holds, as desired.  We claim that $\mathcal L = \varnothing$.  Suppose not, and let $v\in \mathcal L$ such that $|L(v)|$ is minimum.  By the choice of $v$, $\subserv(v) \subseteq S_\infty$.  Since $v$ is very lordly, $|\subserv(v)| \geq \Gap(v)/4$.  By Proposition~\ref{normal Gap linear in save}, $v$ has at least $(\Save_L(v) + \varepsilon\log^{10}k) / (4\eta + 8\eps) \geq (\Save_L(v) + \varepsilon\log^{10}k) / \mu$ neighbors in $S_\infty$, so $v\in S_\infty$, a contradiction.  Hence, $\mathcal L = \varnothing$, as claimed.  Therefore $S_\infty = V(G)$.

  Let $\ordering$ be an ordering of $V(G)$ satisfying the following: If $u\in S_i$ and $v\in S_j$ such that $i > j$, let $u\ordering v$.  By the construction of the sets $S_i$, every vertex $v\in V(G)$ is either aberrant, slightly aberrant, egalitarian-sparse, bipartite-sparse, or $\ordering$-prioritized.  Since $|L(v)| \leq k$ and $(1 - 11\varepsilon)d(v) \leq |L(v)| \leq d(v)$ for each vertex $v$, the graph $G$ is saved with respect to $L$ and $k$, as desired.
\end{proof}

\section{Putting it all together}\label{proof section}

In this section we prove our main technical result, Theorem~\ref{local critical thm}, and we derive Theorems~\ref{nice local critical thm} and \ref{main avg degree bound}.  We restate these results for the reader's convenience.
The proof of Theorem~\ref{epsilon mad reeds} using Theorem~\ref{main avg degree bound} is straightforward and the same argument can be found in~\cite[p.~217]{KP18}, so we omit it.

\techThm*
\begin{proof}
  Let $G$ be an $L$-critical graph for a list assignment $L$ such that for each vertex $v\in V(G)$, we have $|L(v)| \leq k$ where $k$ is sufficiently large as in Theorem~\ref{saved is colorable}, and suppose for a contradiction that~\eqref{main thm equation} does not hold.  Since $G$ is $L$-critical, by~\cite[Theorem 4.1]{KP18}, $G$ does not contain a subgraph that is dense with respect to $L$.  Moreover, for each vertex $v\in V(G)$, we have $|L(v)| \leq d(v)$.  We may assume $G$ is not saved with respect to $L$ and $k$ by Theorem~\ref{saved is colorable}.

  Therefore, by Lemma~\ref{inductive discharging lemma}, there is a nonempty set $D\subsetneq V(G)$ such that
  \begin{equation*}
    \sum_{v\in V(G - D)}\Save_L(v) + \varepsilon\log^{10}k\leq \sum_{v\in V(G - D)}2\varepsilon\Gap_{G - D}(v) - 7\varepsilon(k - |L(v)| + |N(v)\cap D|).
  \end{equation*}
  Subject to that, we choose $D$ to have maximum cardinality.

  We claim that the subgraph $G - D$ is saved with respect to $L$ and $k$.  To that end, suppose $\phi$ is an $L$-coloring of $G[D]$, and let $L'(v) = L(v)\setminus(\cup_{u\in N(v)\cap D}\phi(u))$ for each vertex $v\in V(G - D)$.  Note that $\Save_{L'}(v) \leq \Save_L(v)$ for each vertex $v$.  We assume that equality holds for each vertex, by possibly removing colors from $L'(v)$ arbitrarily.  This assumption also implies that $|L'(v)| = |L(v)| - |N(v) \cap D|$ for each vertex $v$, so \eqref{main thm equation} does not hold for $G - D$ and $L'$.  Thus, $G - D$ and $L'$ satisfy the hypotheses of Lemma~\ref{inductive discharging lemma}.  By the choice of $D$, (b) does not hold.  Hence, by Lemma~\ref{inductive discharging lemma}, $G - D$ is saved with respect to $L'$ and $k$ and thus with respect to $L$ and $k$, as claimed.  
  
  Since $G$ is $L$-critical, there is an $L$-coloring $\phi$ of $G[D]$.  Since $G - D$ is saved with respect to $L'$ and $k$ where $L'(v) = L(v)\setminus(\cup_{u\in N(v)\cap D}\phi(u))$, by Theorem~\ref{saved is colorable}, $G - D$ is $L'$-colorable.  By combining an $L'$-coloring of $G - D$ with $\phi$, we obtain an $L$-coloring of $G$, contradicting that $G$ is $L$-critical.
\end{proof}

Next we show how Theorem~\ref{nice local critical thm} follows from Theorem~\ref{local critical thm}.
\localThm*
\begin{proof}
  Let $G$ be an $L$-critical graph for some list assignment satisfying $|L(v)| = k - 1$ and $\Gap(v) \geq \log^{10}k$ for each $v\in V(G)$ where $k - 1$ is sufficiently large to apply Theorem~\ref{local critical thm}.  By Theorem~\ref{local critical thm} applied with $k - 1$, we have
  \begin{equation*}
    \sum_{v\in V(G)}(\Save_L(v) + \varepsilon\log^{10}(k - 1)) \geq \sum_{v\in V(G)} 2\varepsilon \Gap(v) - 7\varepsilon(k - 1 - |L(v)|).
  \end{equation*}
  Since $\Gap(v) \geq \log^{10}(k - 1)$ and $|L(v)| = k - 1$, the result follows.
\end{proof}

We conclude by proving Theorem~\ref{main avg degree bound}.
\mainThm*
\begin{proof}
  Let $G$ be an $L$-critical graph for some list assignment where $|L(v)| = k - 1$ for each $v\in V(G)$, $\omega(G) \leq k - \log^{10}k$ and $k$ is sufficiently large to apply Theorem~\ref{nice local critical thm}.  Since $G$ is $L$-critical, $d(v) \geq k - 1$ for each $v\in V(G)$.  Therefore for each vertex $v$, since $\omega(v) \leq \omega(G) \leq k - \log^{10}k$, we have $\Gap(v) \geq \log^{10}k$.

  Let $\varepsilon' = \varepsilon/(1 + \varepsilon)$, and note that $\varepsilon' \leq \eps$.  By Theorem~\ref{nice local critical thm},
  \begin{equation*}
    \sum_{v\in V(G)}\Save_L(v) \geq \sum_{v\in V(G)}\varepsilon'\Gap(v).
  \end{equation*}
  Since $\Save_L(v) = d(v) + 1 - (k - 1)$ and $\Gap(v) = d(v) + 1 - \omega(v)$, by rearranging terms in the previous inequality, we obtain $\sum_{v\in V(G)}(1 - \varepsilon')(d(v) + 1) + \varepsilon'\omega(v) \geq (k - 1)|V(G)|.$  Rearranging terms again, we have
  \begin{equation*}
    \ad(G) \geq \frac{k - 1 - \varepsilon'\sum_{v\in V(G)}\omega(v)/|V(G)|}{1 - \varepsilon'} - 1 = (1 + \varepsilon)(k - 1) - \varepsilon\sum_{v\in V(G)}\omega(v)/|V(G)| - 1.
  \end{equation*}
  Since $\sum_{v\in V(G)}\omega(v)/|V(G)| \leq \omega(G)$, the result follows.
\end{proof}

\section*{Acknowledgements}

We thank the anonymous referees for their careful reading of this paper and their suggestions.


\begin{thebibliography}{10}

\bibitem{AH76}
K.~Appel and W.~Haken.
\newblock Every planar map is four colorable.
\newblock {\em Bull. Amer. Math. Soc.}, 82(5):711--712, 1976.

\bibitem{BPP18}
M.~{Bonamy}, T.~{Perrett}, and L.~{Postle}.
\newblock Colouring graphs with sparse neighbourhoods: Bounds and applications.
\newblock {\em J. Combin. Theory Ser. B}, 155:278--317, 2022.

\bibitem{B41}
R.~L. Brooks.
\newblock On colouring the nodes of a network.
\newblock {\em Math. Proc. Cambridge Philos. Soc.}, 37(2):194--197, 1941.

\bibitem{DP17}
M.~Delcourt and L.~Postle.
\newblock On the list coloring version of {Reed's Conjecture}.
\newblock manuscript.

\bibitem{D57}
G.~A. Dirac.
\newblock A theorem of {R}. {L}. {B}rooks and a conjecture of {H}. {H}adwiger.
\newblock {\em Proc. London Math. Soc. (3)}, 7:161--195, 1957.

\bibitem{DP15}
Z.~Dvo\v{r}\'{a}k and L.~Postle.
\newblock Correspondence coloring and its application to list-coloring planar
  graphs without cycles of lengths 4 to 8.
\newblock {\em J. Combin. Theory Ser. B}, 129:38--54, 2018.

\bibitem{G63-1}
T.~Gallai.
\newblock Kritische {G}raphen. {I}.
\newblock {\em Magyar Tud. Akad. Mat. Kutat\'{o} Int. K\"{o}zl.}, 8:165--192,
  1963.

\bibitem{G63-2}
T.~Gallai.
\newblock Kritische {G}raphen. {II}.
\newblock {\em Magyar Tud. Akad. Mat. Kutat\'{o} Int. K\"{o}zl.}, 8:373--395
  (1964), 1963.

\bibitem{HdVK20}
E.~Hurley, R.~de~Joannis~de Verclos, and R.~J. Kang.
\newblock An improved procedure for colouring graphs of bounded local density.
\newblock {\em Adv. Comb.}, Paper No. 7, 33 pp., 2022.

\bibitem{JT95}
T.~R. Jensen and B.~Toft.
\newblock {\em Graph coloring problems}.
\newblock Wiley-Interscience Series in Discrete Mathematics and Optimization.
  John Wiley \& Sons, Inc., New York, 1995.
\newblock A Wiley-Interscience Publication.

\bibitem{KP18}
T.~Kelly and L.~Postle.
\newblock {A local epsilon version of Reed's Conjecture}.
\newblock {\em J. Combin. Theory Ser. B}, 141:181--222, 2020.

\bibitem{KP-corrigendum}
T.~Kelly and L.~Postle.
\newblock Corrigendum to “{A local epsilon version of Reed's
  Conjecture}”[{J. Combin. Theory Ser. B} 141 (2020) 181--222].
\newblock {\em J. Combin. Theory Ser. B}, 2021.

\bibitem{KP-arxiv}
T.~Kelly and L.~Postle.
\newblock A local epsilon version of {Reed's Conjecture}.
\newblock arXiv:1911.02672, 2021.

\bibitem{KS00}
A.~Kostochka and M.~Stiebitz.
\newblock On the number of edges in colour-critical graphs and hypergraphs.
\newblock {\em Combinatorica}, 20(4):521--530, 2000.

\bibitem{KY14}
A.~Kostochka and M.~Yancey.
\newblock Ore's conjecture on color-critical graphs is almost true.
\newblock {\em J. Combin. Theory Ser. B}, 109:73--101, 2014.

\bibitem{KS99}
A.~V. Kostochka and M.~Stiebitz.
\newblock Excess in colour-critical graphs.
\newblock In {\em Graph theory and combinatorial biology ({B}alatonlelle,
  1996)}, volume~7 of {\em Bolyai Soc. Math. Stud.}, pages 87--99. J\'{a}nos
  Bolyai Math. Soc., Budapest, 1999.

\bibitem{K97}
M.~Krivelevich.
\newblock On the minimal number of edges in color-critical graphs.
\newblock {\em Combinatorica}, 17(3):401--426, 1997.

\bibitem{O67}
O.~Ore.
\newblock {\em The four-color problem}.
\newblock Pure and Applied Mathematics, Vol. 27. Academic Press, New
  York-London, 1967.

\bibitem{R98}
B.~Reed.
\newblock $\omega$, ${\Delta}$, and $\chi$.
\newblock {\em J. Graph Theory}, 27(4):177--212, 1998.

\bibitem{R02}
I.~Rivin.
\newblock Counting cycles and finite dimensional {$L^p$} norms.
\newblock {\em Adv. in Appl. Math.}, 29(4):647--662, 2002.

\bibitem{RSST97}
N.~Robertson, D.~Sanders, P.~Seymour, and R.~Thomas.
\newblock The four-colour theorem.
\newblock {\em J. Combin. Theory Ser. B}, 70(1):2--44, 1997.

\end{thebibliography}


\end{document}